\documentclass[12pt]{article}
\usepackage{fullpage}
\usepackage{graphicx}
\usepackage{amssymb,amsmath,amsthm}
\usepackage{amsfonts,amsbsy,dsfont,txfonts}
\usepackage[alphabetic]{amsrefs}
\usepackage[T1]{fontenc}
\usepackage{todonotes}
\usepackage[all]{xy}
\usepackage{subfig}
\usepackage{tikz}
\usepackage{color}
\usepackage{a4wide}
\usepackage{hyperref}
\usepackage{authblk}

\newtheorem{Teo}{Theorem}
\newtheorem{Lema}{Lemma}[section]
\newtheorem{Prop}[Lema]{Proposition}
\newtheorem{Conj}{Conjecture}
\newtheorem{Cor}{Corollary}

\theoremstyle{definition}
\newtheorem{Def}[Lema]{Definition}

\theoremstyle{remark}
\newtheorem{remark}{Remark}
\newcommand*{\dsty}{\displaystyle}
\newcommand*{\fun}[3]{#1: #2 \rightarrow #3}
\newcommand*{\Ends}[1]{\mathrm{Ends}(#1)}
\newcommand*{\Endsg}[1]{\mathrm{Ends}^{*}(#1)}
\newcommand*{\comp}[1]{\mathcal{C}(#1)}
\newcommand*{\nons}[1]{\mathcal{N}(#1)}
\newcommand*{\gs}[1]{\mathcal{G}(#1)}
\newcommand*{\gra}[1]{\mathcal{C}^{1}(#1)}
\newcommand*{\Aut}[1]{\mathrm{Aut}(#1)}
\newcommand*{\adj}[1]{\mathcal{A}(#1)}
\newcommand*{\MCG}[1]{\mathrm{MCG}^{*}(#1)}

\newcommand{\df}[1]{{\bf{#1}}}

\def\N{\mathbb{N}}

\title{Automorphism groups of simplicial complexes of infinite type surfaces}
\author[1]{Jes\'{u}s Hern\'{a}ndez Hern\'{a}ndez}
\author[2]{ Jos\'{e} Ferr\'{a}n Valdez Lorenzo}
\affil[1]{Aix-Marseille Universit\'e, 39 rue F. Joliot Curie, 13453 Marseille Cedex 13, France\\
\emph{e-mail: }\texttt{jhdezhdez@gmail.com} }
\affil[2]{Centro de Ciencias Matem\'aticas,
UNAM, Campus Morelia \\
58190, Morelia, Michoac\'an,  M\'exico\\
\emph{e-mail: }\texttt{ferran@matmor.unam.mx}}

\begin{document}
\maketitle
\begin{abstract}
 \indent Let $S$ be any orientable surface of infinite genus with a finite number of boundary components. In this work we consider the curve complex $\comp{S}$, the nonseparating curve complex $\nons{S}$ and the Schmutz graph $\gs{S}$ of $S$. When all the topological ends of $S$ carry genus, we show that all elements in the  automorphism groups $\Aut{\comp{S}}$, $\Aut{\nons{S}}$ and $\Aut{\gs{S}}$ are \emph{geometric}, \emph{i.e.} these groups are naturally isomorphic to  the \emph{extended} mapping class group $\MCG{S}$ of the infinite surface $S$. Finally,  we study rigidity phenomena within $\Aut{\comp{S}}$ and $\Aut{\nons{S}}$.
\end{abstract}
\section{Introduction}
\indent The mapping class group 
and the extended mapping class group of a given surface $S$, that we will denote by  $\mathrm{MCG}(S)$ and $\mathrm{MCG}^{*}(S)$ respectively, have been studied mostly when $S$ has  \emph{finite topological type}, that is, when its fundamental group is finitely generated. The main purpose of this article is the study of the natural (simplicial) action of the group $\mathrm{MCG}^{*}(S)$ on two abstract simplicial complexes and one simplicial graph associated to $S$ when the surface $S$ has infinite genus. These complexes and graph are:
\begin{enumerate}
\item The  \emph{curve complex} $\comp{S}$. This is the abstract simplicial complex whose vertices are the (isotopy classes of) essential curves in $S$, and whose simplexes are multicurves of finite cardinality. It was introduced by Harvey in 1978.
\item The \emph{nonseparating curve complex} $\nons{S}$. This is the simplicial subcomplex of $\comp{S}$ formed by all nonseparating curves, that is, all the (isotopy classes of) essential curves $\alpha$ such that $S\setminus \alpha$ is connected. This was first introduced by Schmutz in \cite{Sch}.
\item The \emph{Schmutz graph} $\gs{S}$. Introduced by Paul Schmutz Schaller in \cite{Sch}, this is the simplicial graph whose vertex set is the same as the vertex set of $\nons{S}$, and two vertices span an edge whenever their geometric intersection number is 1. It is also known as \emph{a modified complex of nonseparating curves} (see \cite{Farb}), for it can be thought as a 1-dimensional simplicial complex.
\end{enumerate}
Recall that any \emph{orientable} surface of infinite topological type is completely determined, up to homeomorphism, by its genus $g(S)\in\N\cup\{\infty\}$, and a nested pair of topological spaces $\Endsg{S} \subset \Ends{S}$. Roughly speaking, $\Ends{S}$ are the \emph{topological ends} of $S$ and $\Endsg{S}$ is formed by those that carry (infinite) genus. We will focus our attention on infinite genus surfaces $S$ for which the boundary $\partial S$ has finitely many connected components (possibly none) and $\Endsg{S}=\Ends{S}$.\\
\indent To each of the aforementioned simplicial complexes one can associate its automorphism group. We denote these groups by $\Aut{\comp{S}}$, $\Aut{\nons{S}}$ and $\Aut{\gs{S}}$ respectively. For surfaces of finite topological type of positive genus (with the exception of the two-holed torus), every element in $\Aut{\comp{S}}$, $\Aut{\nons{S}}$ and $\Aut{\gs{S}}$ is \emph{geometric}.  That is, if $X= \comp{S}, \nons{S}$ or $\gs{S}$, then the natural map:
\begin{equation}
	\label{E:natmap}
 \begin{tabular}{rrcl}
 $\Psi_{X}:$ & $\MCG{S}$ & $\longrightarrow$ & $\Aut{X}$\\
 & $[h]$ & $\mapsto$ & $h_{*}$
 \end{tabular}
\end{equation}
where $h_{*}$ is given by $h_{*}([\alpha]) = [h(\alpha)]$ is an isomorphism. This result is due to Ivanov \cite{Ivanov} for $X=\comp{S}$, to  Irmak and Schmutz \cite{Irmak}, \cite{Sch} for $X=\nons{S}$  and to Schmutz  \cite{Sch} when $X=\gs{S}$. The main purpose of this article is to extend this result for surfaces of infinite genus:
\begin{Teo}
	\label{TeoISO}
Let $S$ be an infinite genus surface with finitely many boundary components such that $\Endsg{S} = \Ends{S}$. Then the natural map $\Psi_{X}:\MCG{S}\longrightarrow\Aut{X}$ is an isomorphism for $X=\comp{S}, \nons{S}$ or $\gs{S}$.
\end{Teo}
The techniques that we use to prove this result rely heavily on the hypothesis $\Endsg{S} = \Ends{S}$. However, we suspect that this theorem remains valid for surfaces with arbitrarily many planar ends. In addition to the study of the action of $\mathrm{MCG}^{*}(S)$ on simplicial complexes, we study rigidity phenomena within the curve complex and the nonseparating curve complex. More precisely: 
%
\begin{Teo}
	\label{S1homeoS2}
 Let $S_{1}$ and $S_{2}$ be infinite genus surfaces with finitely many boundary components, such that $\Ends{S_{i}} = \Endsg{S_{i}}$ for $i = 1,2$ and let $\fun{\phi}{\comp{S_{1}}}{\comp{S_{2}}}$ be an isomorphism. Then $S_{1}$ is homeomorphic to $S_{2}$.
\end{Teo}
As we will see in section \S \ref{SS:rigidity} this result is not valid if we allow the infinite genus surface $S$ to have planar ends. In \S \ref{SS:rigidity} we will also see that the tools used in the proof of this theorem work for nonseparating curves. Hence, we have the following:
\begin{Cor}
	\label{CorIsoNSC}
 Let $S_{1}$ and $S_{2}$ be infinite genus surfaces with finitely many boundary components, such that $\Ends{S_{i}} = \Endsg{S_{i}}$ for $i = 1,2$ and  let $\fun{\phi}{\mathcal{N}(S_{1})}{\mathcal{N}(S_{2})}$ be an isomorphism. Then $S_{1}$ is homeomorphic to $S_{2}$.
\end{Cor}
As we will see in section \S \ref{S:CFE}, contrary to the compact case, this kind of rigidity results cannot be extended to injective simplicial maps when $S$ is an infinite genus surface. \\
\indent We must remark that while the results of this article are highly inspired by those of the compact case, many proofs have been either modified or outright rewritten to accommodate for the infinite type surfaces. We have also stablished new results and techniques on which the main results of this article rely. Among these we underline the relation between $\Ends{S}$ and the space of ends of the adjacency graph of a pants decomposition of $S$ (see \S \ref{S:adjg}, theorem \ref{GraphEndsSurfaceEnds}) and a variant of the Alexander method for infinite type surfaces (see \S \ref{injCS}, theorem  \ref{GraphEndsSurfaceEnds}).\\
\indent We refer the reader to \cite{Fou1}, \cite{Fou2} and \cite{Fuji} for previous work on groups formed by mapping classes of infinite type surfaces. We want to stress, however, that the cited authors focus their work on several subgroups of what we here call the mapping class group (\emph{e.g.} those with assymptotic qualities for a specific surface or quasiconformal automorphisms of a Riemann surface) and on their action on the Teichm\"uller space.\\[0.3cm]
\emph{Acknowledgements}. We want to thank Camilo Ram\'irez Maluendas for the question that lead to the creation of this article. We are greateful to Hamish Short and Javier Aramayona for carefully reading preliminary versions of this text. The first author would like to thanks Daniel Juan Pineda for his support during the realization of this project. The second author was generously supported by LAISLA, CONACYT CB-2009-01 127991 and PAPIIT projects IN103411 \& IB100212 during the realization of this project.

\section{Preliminaries}

\subsection{Topological invariants for infinite type surfaces}
Let $X$ be a locally compact, locally connected, connected Hausdorff space.

\begin{Def}\cite{F}
\label{D:ends}
Let $U_{1}\supseteq U_{2}\supseteq\ldots$ be an infinite sequence of non-empty connected open subsets of $X$ such that for each $i\in\N$ the boundary $\partial U_{i}$ is compact and $\bigcap\limits_{i\in\N}\overline{U_{i}}=\emptyset$. Two such sequences $U_{1}\supseteq U_{2}\supseteq\ldots$ and $U'_{1}\supseteq U'_{2}\supseteq\ldots$ are said to be equivalent if for every $i\in\N$ there exist $j,k$ such that $U_{i}\supseteq U'_{j}$ and $U'_{i}\supseteq U_{k}$. The corresponding equivalence class is called a \df{topological end} of $X$.  
\end{Def}
The set of ends ${\rm Ends}(X)$ of $X$ can be endowed with a topology in the following way. For any set $U$ in $X$ whose boundary is compact, we define $U^{*}$ to be the set of all ends $[U_{1}\supseteq U_{2}\supseteq\ldots]$ for which there is a representative such that $U_{n}\subset U$ for $n$ sufficiently large. With respect to this topology, 
${\rm Ends}(X)$ is a compact, closed, totally disconnected space without interior points (see for example Theorem 1.5, \cite{Ray}).\\
\indent The \textbf{genus} of a surface is the maximum of the genera of its compact subsurfaces.  A surface is said to be \textbf{planar} if all of its compact subsurfaces are of genus zero. We define $\Endsg{S}\subset{\rm Ends}(S)$ as the set of all ends which are not planar. As stated in the following theorem, any orientable surface is determined, up to homeomorphism, by its genus, boundary and space of ends. Henceforth all surfaces in this text are connected. 
\begin{Teo} 
	\label{T:NCSclass}
	Let $S$ and $S'$ be two orientable surfaces of the same genus. Then $S$ and $S'$ are homeomorphic if and only if they have the same number of boundary components, and $\Endsg{S}\subset{\rm Ends}(S)$ and $\Endsg{S'}\subset{\rm Ends}(S')$ are homeomorphic as nested topological spaces.
\end{Teo}
The proof of this theorem for the case when $S$ and $S'$ have no boundary can be found in \cite{R}. The case for surfaces with boundary was proven in \cite{PM}.

\subsection{Complexes and graphs of curves}
	\label{SS:CC}
\indent There are several curve complexes that one can associate to a surface of finite genus, with finitely many boundary components and punctures. In this section we extend the definitions of these complexes to noncompact surfaces of infinite topological type and explore some of their basic properties.\\
\indent Abusing language and notation, we will call \textit{curve},  a topological embedding $S^{1} \hookrightarrow S$,  the isotopy class of this embedding and its image on $S$. A curve is said to be \emph{essential} if it is neither homotopic to a point nor to a boundary component. Hereafter all curves are considered essential unless otherwise stated. An essential curve is said to be \emph{separating} if the surface obtained by cutting $S$ along its image is disconnected. It is said to be nonseparating otherwise. A separating curve $\alpha$ is said to be an \emph{outer separating} curve if by cutting $S$ along $\alpha$ one of the resulting connected components is a pair of pants (\emph{i.e.} a genus 0 surface with three boundary components). A non-outer separating curve is a separating curve which is not an outer separating curve. Two curves are \emph{disjoint} if they are distinct and their (geometric) intersection number is $0$.

\begin{Def} [\sc Multicurves]
A multicurve is either a set of just one curve, or a pairwise disjoint and locally finite set of curves of $S$. We allow multicurves to consist of an infinite set of curves. If $M$ is a multicurve of $S$, the surface obtained by cutting $S$ along pairwise disjoint representatives of the elements of $M$ will be denoted by $S_{M}$.
\end{Def}
Infinite \emph{countable} multicurves arise in surfaces with nonfinitely generated fundamental group. Take for example the Loch  Ness Monster, that is,  a surface with infinite genus and one end. If $S$ is a compact surface of genus $g$ with $n$ boundary components, the \emph{complexity} of $S$, denoted by $\kappa(S)$, is equal to $3g-3+n$. This is the cardinality of a maximal multicurve in $S$.

\begin{Def} [\sc The Curve Complex]
The Curve complex of $S$, $\comp{S}$, is the abstract simplicial complex whose vertices are the isotopy classes of essential curves in $S$, and whose simplexes are multicurves of finite cardinality. We denote the set of vertices of $\comp{S}$ by $\mathcal{V}(\comp{S})$.
\end{Def}
The $1$-skeleton of $\comp{S}$ will be denoted by $\gra{S}$. Since every automorphism of $\comp{S}$  is determined uniquely by a function of its vertices, and the same statement is true for automorphisms of $\gra{S}$, then the groups $\Aut{\comp{S}}$ and $ \Aut{\gra{S}}$ are isomorphic.
\begin{Def} [\sc The Nonseparating Curve Complex]
 The Nonseparating curve complex of $S$, $\nons{S}$, is the subcomplex of $\comp{S}$ whose vertices are the isotopy classes of essential \emph{nonseparating} curves in $S$. We denote the set of vertices of $\nons{S}$ by $\mathcal{V}(\nons{S})$.
\end{Def}
\begin{Def} [\sc The Schmutz graph]
 The Schmutz graph of $S$, $\gs{S}$, is the simplicial graph whose vertices are the isotopy classes of essential nonseparating curves in $S$, and two vertices span an edge if their geometric intersection number is $1$.
\end{Def}
\begin{Prop}
	\label{T:HypCC}
 Let $S$ be a surface of infinite genus. Then $\comp{S}$, $\nons{S}$ and $\gs{S}$ are connected.  In particular $\gra{S}$ and $\mathcal{N}^{1}(S)$ have diameter 2 while $\gs{S}$
 has diameter 4.
\end{Prop}
\begin{proof}
Given any two distinct curves $\alpha$ and $\beta$ (either in $\mathcal{V}(\comp{S})$ or in $\mathcal{V}(\nons{S})$), we can always find a compact (finite genus) subsurface $S^{\prime}$ such that contains $\alpha$ and $\beta$. Hence we can take an essential nonseparating curve $\gamma$ on $S$ contained in $S \backslash S^{\prime}$ and not isotopic to $\alpha$ and $\beta$. Therefore $\gra{S}$ and $\mathcal{N}^{1}(S)$ are connected, $\mathrm{diam}(\gra{S}) = \mathrm{diam}(\mathcal{N}^{1}(S)) = 2$.\\
\indent If $\alpha$ and $\beta$ are two distinct nonseparating curves, as in the paragraph above, we can always find a curve $\gamma$ such that $i(\alpha,\gamma) = i(\gamma,\beta) = 0$; then we can always find curves $\delta_{1}$ and $\delta_{2}$ such that $i(\alpha,\delta_{1}) = i(\delta_{1},\gamma) = i(\gamma,\delta_{2}) = i(\delta_{2},\beta)=1$. Hence $\gs{S}$ is connected, $\mathrm{diam}(\gs{S}) \leq 4$.
\end{proof}

\begin{remark}
Number 2 as diameter for $\gra{S}$ and $\mathcal{N}^{1}(S)$ is optimal, but 4 as diameter for $\gs{S}$ is not necessarily optimal. 
\end{remark}

\subsection{Mapping Class Group}
\indent Through this article, we will be working with the mapping class group of a surface $S$. When $S$ is compact, this group has different (equivalent) definitions, see for example \cite{Farb}, \S 2.1. In this paper we will be working with the following definition.
\begin{Def}[\sc Mapping Class Group]
 Let $S$ be a surface. Then $\mathrm{Homeo}^{+}(S,\partial S)$ is the group of orientation-preserving homeomorphisms of $S$ that restrict to the identity on the boundary, and $\mathrm{Homeo}(S)$ is the group of \emph{all} homeomorphisms of $S$. The mapping class group of $S$, $\mathrm{MCG}(S)$ is the group $\mathrm{Homeo}^{+}(S) / \thicksim$, where $\thicksim$ represents the isotopy relation relative to the boundary. The extended mapping class group of $S$ is the group $\MCG{S} \coloneqq \mathrm{Homeo}(S) / \thicksim$, where $\thicksim$ represents the isotopy relation.
\end{Def}
The group $\MCG{S}$  is incredibly big. As evidence for this we have the following lemma and corollaries.
\begin{Lema}
 Let $S$ be an infinite genus surface and $F$ a subsurface of $S$ such that $S \backslash F$ has genus at least $1$ and the boundary components of $F$ are either boundary components of $S$ or essential curves of $S$. Then there exists a subgroup of $\MCG{S}$ isomorphic to $\mathrm{MCG}(F)$, with infinite index in $\MCG{S}$.
\end{Lema}
\begin{proof}
 The subgroup of $\MCG{S}$ formed by those orientation-preserving elements $[h] \in \MCG{S}$ that have a representative $h$ with support on $F$, is isomorphic to $\mathrm{MCG}(F)$. This subgroup will have index greater or equal to the number of different elements in $\MCG{S}$ that have its support in the interior of the complement of $F$, thus it will have infinite index.
\end{proof}
\begin{Cor}
 Let $S$ be an infinite genus surface and $S_{g,n}$ be a compact surface of genus $g$ and $n$ boundary components. Then $\MCG{S_{g,n}} < \MCG{S}$.
\end{Cor}
\begin{Cor}
  Let $S$ be an infinite genus surface, $\{(g_{i},n_{i})\}_{i \in \mathbb{N}} \subset (\mathbb{N} \times \mathbb{Z}^{+}) \backslash \{(0,1)\}$ be a sequence and $S_{i}$ be a compact orientable surface of genus $g_{i}$ and $n_{i}$ boundary components. Then $\MCG{S}$ contains a subgroup isomorphic to $\prod_{i \in \mathbb{N}}\MCG{S_{i}}$.
\end{Cor}

\section{Ends of adjacency graphs and surfaces}
	\label{S:adjg}
In this section we prove that, under the hypotheses $\Ends{S} = \Endsg{S}$, one can determine topologically $\Ends{S}$ using the adjacency graph of a pants decomposition of $S$.
\begin{Def}[\sc Pants decomposition and the adjacency graph]
 A \emph{pants decomposition} is a multicurve $P$ of maximal cardinality. We say $\alpha,\beta \in P$ are adjacent with respect to $P$ if they bound the same pair of pants in $S_{P}$. The adjacency graph of $P$, $\adj{P}$, is the simplicial graph whose vertex set is $P$ and two vertices span an edge if and only if they are adjacent with respect to $P$. We say two nonseparating curves form a \emph{peripheral pair} if they bound, along with a boundary component of $S$, a pair of pants.
\end{Def}
\indent If $P$ is a pants decomposition,  $S_{P}$ is the disjoint union of surfaces homeomorphic to a pair of pants, for otherwise we contradict maximality. As an abstract graph, $\adj{P}$ is a subgraph of $\gra{S}$, but we have to keep in mind that
adjacency of vertices in $\adj{P}$ and $\gra{S}$  means different things for the corresponding curves in $S$.

\begin{remark}\label{Cutpointnonouter}
 It can be easily checked that the only cut points of an adjacency graph $\adj{P}$ are non-outer separating curves, and non-outer separating curves are always cut points of any adjacency graph in which they are vertices. Also, we can easily check outer separating curves always have degree less or equal to two.
\end{remark}
\begin{Teo}
	\label{GraphEndsSurfaceEnds}
Let $S$ be an infinite genus surface such that $\Ends{S} = \Endsg{S}$ and $P$ be a pants decomposition of $S$. Then $\Ends{\adj{P}}$ is homeomorphic to $\Ends{S}$.
\end{Teo}
\begin{proof} For every pants decomposition there is a natural, but not canonical, topological embedding:
\begin{equation}
	\label{E:emb}
f:\adj{P}\hookrightarrow S
\end{equation} 
This embedding is illustrated in figure \ref{F:ooo}. Let $\Gamma$ be a subgraph of $\adj{P}$ whose boundary $\partial\Gamma$ is compact. We define $S(\Gamma)$ as the subsurface of $S$ formed by all pants in $S$ (defined by the multicurve $P$) that intersect $f(\Gamma)$, \emph{deprived of its boundary}. By definition $S(\Gamma)$
is an open subsurface of $S$ whose boundary is formed by a finite collection of curves $\{C_{1},\ldots,C_{n}\}\subset P$. Remark that if the graph $\Gamma$ is connected, so is $S(\Gamma)$. Moreover if $\Gamma\supset\Gamma'$ are two connected subgraphs of $\adj{P}$ with compact boundaries we have that $S(\Gamma)\supset S(\Gamma')$. For every  $[\Gamma_{1}\supseteq \Gamma_{2}\supseteq\ldots]$ in $\Ends{\adj{P}}$ we define:
\begin{equation}
f_{*}[\Gamma_{1}\supseteq \Gamma_{2}\supseteq\ldots]=[S(\Gamma_{1})\supseteq S(\Gamma_{2})\supseteq\ldots]\in \Ends{S}
\end{equation}
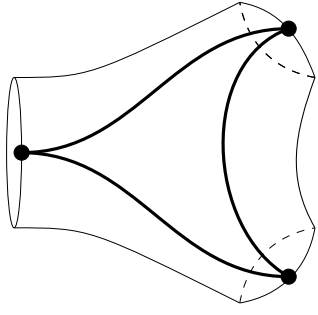
\begin{figure}
\centering
\begin{tikzpicture}
\draw [very thin] (0,1) ellipse (0.1 and 1);
\draw [very thin] (0,2) .. controls (1,2) .. (3,3);
\draw [very thin] (0,0) .. controls (1,0) .. (3,-1);
\draw [very thin] (3,3) to[out=-10,in=110] (4,2) ;
\draw [dashed] (4,2) to[out=170,in=280] (3,3) ;
\draw [very thin, dashed] (3,-1) to[out=80,in=190] (4,0) ;
\draw [very thin] (3,-1) to[out=10,in=260] (4,0) ;
\draw (4,2) to[out=250,in=120] (4,0);
\draw [dashed] (4,2) to[out=170,in=280] (3,3) ;
\draw [fill=black] (0.1,1) circle (0.1); 
\draw [fill=black] (3.65,2.65) circle (0.1);
\draw [fill=black] (3.65,-.65) circle (0.1);
\draw [very thick] (0.1,1) to[in=180,out=0] (3.65,2.65);
\draw [very thick] (0.1,1) to[in=180,out=0] (3.65,-.65);
\draw [very thick] (3.65,-.65) to[in=200,out=150] (3.65,2.65);
\end{tikzpicture}
\caption{A natural embedding of $\adj{P}$ into $S$.} \label{F:ooo}
\end{figure}
It follows directly from definition \ref{D:ends} that $f_{*}$ is well defined. We claim that $f_{*}:\Ends{\adj{P}}\to\Ends{S}$ is an homeomorphism. The injectivity of $f_{*}$ follows from the following general lemma:
\begin{Lema}\cite{F}
Let $[U_{1}\supseteq U_{2}\supseteq\ldots]$ and $[U'_{1}\supseteq U'_{2}\supseteq\ldots]$ be two different points in $\Ends{X}$. Then there exists $i\in\N$ such that $U_{i}\cap U'_{i}=\emptyset$. 
\end{Lema}
Indeed, let us suppose that $[\Gamma_{1}\supseteq \Gamma_{2} \supseteq\ldots] \neq [\Gamma'_{1}\supseteq \Gamma'_{2} \supseteq \ldots]$. Then there exists an $i\in \N$ such that $\Gamma_{i}\cap\Gamma'_{i}=\emptyset$. Let us suppose that $[S(\Gamma_{1})\supseteq S(\Gamma_{2})\supseteq\ldots]=[S(\Gamma'_{1})\supseteq S(\Gamma'_{2})\supseteq\ldots]$. Hence, for the previous $i\in\N$ there exist $l,k\in\N$ such that $S(\Gamma_{i})\supseteq S(\Gamma'_{l})$ and $S(\Gamma'_{i})\supseteq S(\Gamma_{k})$. Without loss of generality suppose that $S(\Gamma'_{l})\supseteq S(\Gamma'_{i})$, hence $S(\Gamma_{i})\cap S(\Gamma'_{i})=S(\Gamma'_{i})$ and, since both $\partial\Gamma_{i}$ and $\partial\Gamma'_{i}$ have compact boundary we conclude that $\Gamma_{i}\cap\Gamma'_{i}\neq\emptyset$. This contradicts our initial assumption. The case where $S(\Gamma'_{i})\supseteq S(\Gamma'_{l})$ is analogous.\\ 
We address now surjectivity. Let $[S_{1}\supseteq S_{2}\supseteq\ldots]\in\Ends{S}$. Since there are no planar ends, that is $\Ends{S} = \Endsg{S}$, we can consider, for each $S_{i}$ the surface $\mathfrak{S}_{i}$ formed by all pants in the pants in the decomposition defined by $P$ that intersect $S_{i}$. Since $S_{i}$ is connected, then $\mathfrak{S}_{i}$ must be connected. Also, since $\partial S_{i}$ is compact, so is $\partial\mathfrak{S}_{i}$. Moreover, by definition, if $i\leq j$ then $\mathfrak{S}_{i}\supseteq\mathfrak{S}_{j}$. Hence we have a well defined end $[\mathfrak{S}_{1} \supseteq \mathfrak{S}_{2} \supseteq \ldots]\in\Ends{S}$. By construction, for every $i\in\N$ we have that $S_{i}\subset\mathfrak{S}_{i}$.  On the other hand, given $i\in\N$ we can find $S_{j}$ such that $S_{i}\setminus S_{j}$ contains (properly) a connected surface formed by pants in the pant decomposition defined by $P$. This implies that there exists $k\in\N$ such that $S_{j}\subset \mathfrak{S_{k}}$. Therefore $[\mathfrak{S}_{1} \supseteq \mathfrak{S}_{2} \supseteq \ldots]=[S_{1}\supseteq S_{2}\supseteq\ldots]$. Now define $\Gamma_{i}$ as the maximal subgraph of $S$ such that $f(\Gamma_{i})\subset\mathfrak{S}_{i}$. The graph $\Gamma_{i}$ has compact boundary for $\mathfrak{S}_{i}$ has compact boundary and by definition $f_{*}[\Gamma_{1}\supset\Gamma_{2}\supset\ldots]=[S_{1}\supseteq S_{2}\supseteq\ldots]$. This proves that $f_{*}$ is a bijection.\\
Now we prove that $f_{*}$ is an homeomorphism. Let $\Gamma$ be a subgraph of $\adj{P}$ with compact boundary as before. We define
\begin{equation}
\Gamma^{*} \coloneqq \{[\Gamma_{1} \supseteq \Gamma_{2} \supseteq \ldots]\hspace{1mm}|\hspace{1mm} \Gamma \supseteq \Gamma_{i}\hspace{1mm}\text{for i sufficiently big}\}
\end{equation}
The collection of all $\Gamma^{*}$'s generates the topology of $\Ends{\adj{P}}$. On the other hand we know from \cite{R} that the topology of $\Ends{S}$ is generated by 
\begin{equation}
U^{*}:=\{[\hat{S}_{1}\supseteq \hat{S}_{2}\supseteq\ldots]\hspace{1mm}|\hspace{1mm} U\supseteq \hat{S}_{i}\hspace{1mm}\text{for i sufficiently big}\},
\end{equation}
where $U\subset S$ is an open subset with compact boundary. Clearly $f_{*}\Gamma_{*}=S(\Gamma)^{*}$ , hence $f_{*}$ is open. From \cite{Ray} we know that both $\Ends{\adj{P}}$ and $\Ends{S}$ are compact Hausdorff topological spaces. Hence $f_{*}$ is an homeomorphism.
\end{proof}
\begin{remark}
	\label{Rmk:pun}
We can think of punctures on a surface as planar ends, and hence the preceding result is not true if we allow the surface $S$ to have them. 
\end{remark}

\section{Proof of main results.}
	\label{MainResults}
	\subsection{Injectivity.}
		\label	{S:INYECT}
\indent In this section we will prove the following result:
\begin{Teo}\label{injCS}
 Let $S$ be an infinite genus surface such that
$\Ends{S} = \Endsg{S}$. The natural map: 
\begin{equation}
\Psi_{\comp{S}}:{\rm MCG}^{*}(S)\to {\rm Aut}(\comp{S})
\end{equation}
is injective.
\end{Teo}
Most of the proof of this theorem will rely in the following lemma and a variant of the Alexander method (see \cite{Farb} for details on this method). 
\begin{Lema}
\label{homotopycurves}
 Let $S$ be an infinite genus surface possibly with marked points and possibly a finite number of boundary components. Let $\gamma_{1}, \ldots, \gamma_{n}$ be a collection of simple closed curves and simple proper arcs in $S$ such that satisfy the three following properties:
 \begin{enumerate}
  \item The $\gamma_{i}$ are in pairwise minimal position. That is, for $i\neq j$, the 
(geometric) intersection of $\gamma_{i}$ with $\gamma_{j}$ is minimal within their homotopy classes. 
\item The $\gamma_{i}$ are pairwise nonisotopic.
  \item For distinct $i,j,k$, at least one of $\gamma_{i} \cap \gamma_{j}$, $\gamma_{i} \cap \gamma_{k}$, or $\gamma_{j} \cap \gamma_{k}$ is empty.
 \end{enumerate}
 If $\gamma_{1}^{\prime}, \ldots, \gamma_{n}^{\prime}$ is another such collection so that $\gamma_{i}$ is isotopic to $\gamma_{i}^{\prime}$ for each $i$, then there is an isotopy of $S$  that takes $\gamma_{i}^{\prime}$ to $\gamma_{i}$ for all $i$ simultaneously, and hence takes $\cup \gamma_{i}$ to $\cup \gamma_{i}^{\prime}$.
\end{Lema}
A collection of curves $\gamma_{1}, \ldots, \gamma_{n}$ satisfying (1)-(3) in the preceding lemma will be called an \emph{Alexander system} in S. The proof of this lemma is exactly the same as the proof of lemma 2.9 in \cite{Farb}.\\[0.3cm]
\textbf{Proof theorem \ref{injCS}}. Let $\fun{h}{S}{S}$ be an homeomorphism such that $h(\alpha)$ is isotopic to $\alpha$ for all $\alpha \in \mathcal{V}(\comp{S})$. For every infinite genus surface such that
$\Ends{S} = \Endsg{S}$ we can find a family of compact subsurfaces  $\{K_{i}\}_{i \in \mathbf{N}}$ such that:
\begin{itemize}
\item $S=\bigcup_{i\in \mathbf{N}}K_{i}$,
\item $K_{i} \subset K_{j}$ if $i < j$.
\item $K_{i}$ has genus at least $3$ for all $i \in \mathbf{N}$.
\item $K_{j}\setminus K_{i}$ admits at least one curve nonisotopic to any boundary curve of $K_{j}$ for $i < j$. 
\item Every boundary component of $K_{i}$ that is not a boundary component of $S$ is an essential separating curve of $S$.
\end{itemize}
For each $i\in \mathbf{N}$ let us write $\partial K_{i}$ for the boundary of $K_{i}$, $\partial_{S}K_{i}$ for all curves in  $\partial K_{i}$ that are part of the boundary of $S$ and $\partial_{i}K_{i}$ for $\partial K_{i}\setminus\partial_{S}K_{i}$. Given such a family $\{K_{i}\}_{i \in \mathbb{N}}$ of compact subsurfaces we can find $\{\Gamma_{i}\}_{i\in \mathbb{N}}$ a collection of finite subsets of $\mathcal{V}(\comp{S})$ such that:
\begin{itemize}
 \item Every boundary component of $K_{i}$ that is not a boundary component of $S$, is in $\Gamma_{j}$ for $i < j$ and is disjoint from every other curve in $\cup_{i\in \mathbf{N}}\Gamma_{i}$. 
 \item $\Gamma_{0}$ fills $K_{0}$ and $\Gamma_{j}\setminus\Gamma_{j-1}$ fills $K_{j}\setminus K_{j-1}$ for all $j > 0$. In addition $\Gamma_{i} \subset \Gamma_{j}$ for $i < j$.
 \item If we cut $K_{j} \setminus K_{i}$ along $\Gamma_{j} \setminus \Gamma_{i}$ we obtain either discs or annuli with one boundary component in $\partial K_{k}$, for $i < j$ and some $k$ with $i \leq k \leq j$. 
\item For all $\gamma \in (\Gamma_{j} \backslash \Gamma_{i})$ and $\gamma^{\prime} \in \Gamma_{i}$, we have that $i(\gamma,\gamma^{\prime}) = 0$. Moreover, if we define for each $i\in \mathbf{N}$
\begin{equation}
\Gamma_{i}'= \Gamma_{i}\cup\partial_{i}K_{i}
\end{equation}
 then, for all $\gamma \in (\Gamma_{j} \backslash \Gamma_{i}')$ and $\gamma^{\prime} \in \Gamma_{i}'$ we have $i(\gamma,\gamma^{\prime}) = 0$.
 \item Both $\Gamma_{i}$ and $\Gamma_{i}'$ are Alexander systems in $S$.
\end{itemize}
Figure \ref{GammaiKi} shows an example of  $\{K_{i}\}_{i \in \mathbf{N}}$ and its corresponding $\{\Gamma_{j}\}_{j\in \mathbf{N}}$.

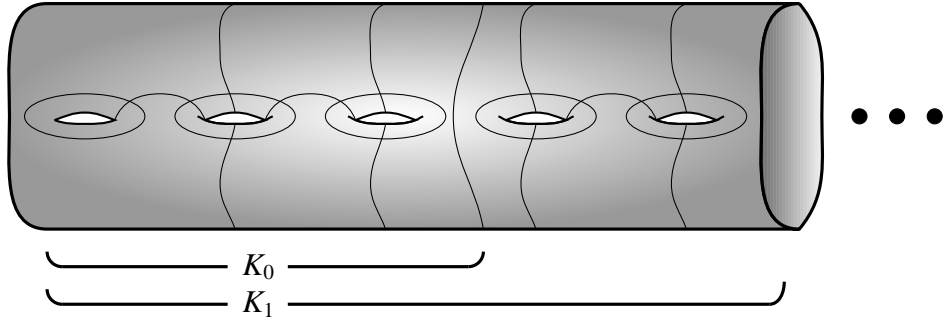
\begin{figure}
\centering
\begin{tikzpicture}
\shadedraw [inner color=white,outer color=white!60!black, very thick]  
(-4,0) to[out=180,in=270] (-4.5,1.5) to[out=90,in=180] (-4,3) to[out=0,in=180] (6,3) to[out=180, in=90] (5.5,1.5) to[out=270,in=180] (6,0) to[out=180,in=0] (-4,0) ;

\shadedraw [left color=white!35!black, right color=white, very thick] 
(6,3) to[out=180, in=90] (5.5,1.5) to[out=270,in=180] (6,0) to[out=50,in=270] (6.3,1.5) to[out=90,in=310] (6,3);


\foreach \x in {.5,1,1.5}
\filldraw (6.3+\x,1.5) circle (.1);


\foreach \x in {0,2,4,6,8}
\filldraw[fill=white, thick] (-3.9+\x,1.45) to[out=340,in=180] (-3.5+\x, 1.4) to [out=0,in=200] (-3.1+\x,1.45) to[out=160,in=0] (-3.5+\x,1.55) to[out=180,in=20] (-3.9+\x,1.45);
\foreach \x in {2,4,6,8}
\draw[thick] (-4+\x,1.5) to[out=340,in=160] (-3.9+\x,1.45) to[out=340,in=180] (-3.5+\x, 1.4) to [out=0,in=200] (-3.1+\x,1.45) to[out=20,in=200] (-3+\x,1.5);


\foreach \x in {0,2,4,6,8}
\draw (-3.5+\x,1.5) ellipse (.8
 and .3);
 \foreach \x in {2,4,6,8}
 \draw (-3.5+\x,1.55) to[out=100,in=270] (-3.7+\x,2.3) to[out=90,in=200] (-3.5+\x,3);

 
 \foreach \x in {0,2,6}
 \draw [thin] (-3.1+\x,1.45) to[out=70,in=180] (-2.5+\x,1.8) to[out=0,in=110](-1.9+\x,1.45);
 
 
 \foreach \x in {2,4,6,8}
 \draw (-3.5+\x,1.4) to[out=260,in=90] (-3.7+\x,.6) to[out=270,in=110] (-3.5+\x,0);
 
 \foreach \x in {2}
 \draw[thin] (-.2+\x,0) to[out=100,in=270] (-.6+\x,1.5) to[out=90,in=260] (-.2+\x,3);

\draw[very thick] (1.6,-.5) -- (-.8,-.5) node[anchor=east]{$K_{0}$};
\draw[very thick] (1.6,-.5) to[out=0,in=270] (1.8,-.3);
\draw[very thick] (-3.8,-.5) -- (-1.6,-.5);
\draw[very thick] (-3.8,-.5) to[out=180,in=270] (-4,-.3);

\foreach \x in {4}
\draw[very thick] (1.6+\x,-1) -- (-.8,-1) node[anchor=east]{$K_{1}$};
\foreach \x in {4}
\draw[very thick] (1.6+\x,-1) to[out=0,in=270] (1.8+\x,-.7);
\draw[very thick] (-3.8,-1) -- (-1.6,-1);
\draw[very thick] (-3.8,-1) to[out=180,in=270] (-4,-.8);

\end{tikzpicture}
\caption{Example for $K_{0}, K_{1}, \ldots$ and $\Gamma_{0}, \Gamma_{1},\ldots$.} \label{GammaiKi}
\end{figure}

\begin{Lema}
	\label{L:HOM}
There exist homotopies $\fun{H_{i}}{S \times [0,1]}{S}$ such that:
\begin{enumerate}
 \item $H_{i}|_{S \times \{0\}}$ is the identity for $i \in \mathbf{N}$.
 \item $H_{i}|_{K_{i} \times \{1\}} = h|_{K_{i}}$ for $i \in \mathbf{N}$.
 \item $H_{i}|_{K_{i} \times [0,1]} = H_{j}|_{K_{i} \times [0,1]}$ for $i < j$.
\end{enumerate}
\end{Lema}
The proof of this lemma is rather technical, the main difficulty being to prove that $H_{i}|_{K_{i} \times [0,1]} = H_{j}|_{K_{i} \times [0,1]}$ for $i < j$.  We leave it for later. We will use the lemma to finish the proof of theorem \ref{injCS}.\\
\indent For every $x\in S$ there exist $i\in \mathbf{N}$ such that
$x \in K_{i}$ and $x \notin \partial_{i}K_{i}$. Define $\fun{H}{S \times [0,1]}{S}$ as $H(x,t) = H_{i}(x,t)$. From (3) in the preceding lemma we deduce that 
 $H$ is well-defined. The function $H$ is clearly continuous, $H|_{S \times \{0\}}$ is the identity and $H|_{S \times \{1\}} = h$. Thus $H$ is an homotopy from the identity to $h$. This, modulo the proof of lemma \ref{L:HOM}, finishes the proof of theorem \ref{injCS}.\\[0.3cm]

\textbf{Proof of lemma \ref{L:HOM}}. The idea of the proof is a variant of the Alexander method (see \cite{Farb} for details on this method). By hypothesis, for every $\gamma\in \comp{S}$ , the curves $\gamma$ and $h(\gamma)$ are isotopic. Using lemma \ref{homotopycurves} we can assure the existence, for each $i\in \mathbf{N}$, of  an isotopy $\fun{\tilde{H}_{i}}{S \times [0,1]}{S}$, that takes $\gamma$ to $h(\gamma)$ for all $\gamma\in\Gamma_{i}'$ simultaneously.  Moreover, since for all $\gamma \in (\Gamma_{j} \backslash \Gamma_{i}')$ and $\gamma^{\prime} \in \Gamma_{i}'$ we have $i(\gamma,\gamma^{\prime}) = 0$,  we can ask 
\begin{equation}
	\label{E:GC}
\dsty \tilde{H}_{i|\hspace{1mm}_{\Gamma_{i}'\times [0,1]}} = \tilde{H}_{j|\hspace{1mm}_{\Gamma_{i}'\times [0,1]}},\hspace{8mm}\text{for $i < j$}.
\end{equation}
In other words, the homotopies can be chosen so that $\tilde{H}_{i}$ moves the curves in $\Gamma_{i}'$ at exactly the same time as $\tilde{H}_{j}$ moves the curves in $\Gamma_{i}'$ for $i < j$. Let us define $f_{i}:={H}_{i| S\times \{1\}}$. Remark that $h^{-1}\circ f_{i}$ fixes all the points in  $\Gamma_{i}'$. On the other hand, $h$ has to be orientation-preserving, since otherwise for every compact subsurface  $S'\hookrightarrow S$ we could find an homeomorphism that reverses orientation and at the same time acts trivially on $\comp{S'}$, which is not possible if $S'$ has genus bigger than 3 and at least one boundary component. Hence $h^{-1}\circ f_{i}$ is orientation-preserving and, by the same argument used by Farb and Margalit (see proof proposition 2.8, p. 62-63, \cite{Farb}), we have that $h^{-1}\circ f_{i}$ sends each connected region in $S\setminus\Gamma_{i}'$ to itself.
 By hypotheses $\Gamma_{0}$ fills $K_{0}$ and $\Gamma_{j}\setminus\Gamma_{j-1}$ fills $K_{j}\setminus K_{j-1}$ for all $j >1$. Hence:
\begin{equation}
S\setminus\Gamma_{i}'=\left(\bigsqcup_{k=1}^{n_{i}} A_{k}\right)\sqcup \left(\bigsqcup_{k=1}^{m_{i}} D_{k}\right)\sqcup S_{i}
\end{equation}
where each $D_{k}$ is homeomorphic to a disc, each $A_{k}$ is homeomorphic to an annulus and $S_{i}=S\setminus K_{i}$ is an infinite genus surface.  Furthermore:
\begin{enumerate}
\item The boundary of each disc $D_{k}$ is formed by segments contained in $\Gamma_{i}$.
\item The boundary of each annulus $A_{k}$ is either contained in $\Gamma_{i}'$ or one of its connected components is also a connected component of the boundary of $S$.
\end{enumerate}
From Alexander's lemma, we deduce that $h^{-1}\circ f_{i}$ restricted to $D_{k}$ is isotopic to $Id_{| D_{k}}$. 
When $A_{k}$ shares a boundary component with $S$, 
the restriction of $h^{-1}\circ f_{i}$ to $A_{k}$
is isotopic to the identity, for we are allowed to perform isotopies on $A_{k}$ that do not fix the boundary of $S$ pointwise. Finally, when $A_{k}$ shares no boundary component with the boundary of $S$ the restriction of $h^{-1}\circ f_{i}$ to $A_{k}$ is also isotopic to the identity for else this restriction will be a non-trivial Dehn twist and we could then find a curve $\gamma\in\comp{S}$ intersecting the interior of $A_{k}$ which is not fixed by $h^{-1}$. From this three facts we conclude that $h^{-1}\circ f_{i}$ is isotopic to the identity in $K_{i}$ and hence $f_{i}$ is isotopic to $h$ in $K_{i}$. The composition of these two isotopies form the desired isotopy $H_{i}$.\qed

\subsection{Rigidity.}
	\label{SS:rigidity}
\indent In this section we give the proof of theorem \ref{S1homeoS2} and corollary \ref{CorIsoNSC}.  This requires some auxiliary  facts and lemmas, that we state and prove in the following paragraphs.\\
\indent Through this section 
 $S_{1}$ and $S_{2}$ will denote (connected) infinite genus surfaces with a finite number of boundary components and $\fun{\phi}{\comp{S_{1}}}{\comp{S_{2}}}$ an isomorphism. We remark that the image via $\phi$ of any pants decomposition of $S_{1}$ is a pants decomposition of $S_{2}$. Moreover, if $P$ is a pants decomposition of $S_{1}$, then $\alpha, \beta \in P$ are adjacent with respect to $P$ if and only if $\phi(\alpha)$ and $\phi(\beta)$ are adjacent with respect to $\phi(P)$.
The sufficiency of this statement can be found in \cite{shackleton} and the necessity follows from the fact that we are dealing with an isomorphism of the curve complex. Therefore  $\fun{\phi}{\comp{S_{1}}}{\comp{S_{2}}}$ induces a map
\begin{equation}
	\label{E:inducediso}
\fun{\varphi}{\adj{P}}{\adj{\phi(P)}}
\end{equation}
as follows: $\alpha \mapsto \varphi(\alpha) := \phi(\alpha)$. Moreover, $\varphi$ is an isomorphism. For this reason cut points of $\adj{P}$ go to cut points under $\phi$ and this isomorphism sends:
 \begin{enumerate}
  \item Non-outer separating curves  to non-outer separating curves.
  \item Nonseparating curves to nonseparating curves.
  \item Outer curves  to outer curves.
 \end{enumerate}
 The proof of (1) and (2) can be found in \cite{shackleton}, where as (3) follows from (1), (2) and the fact that $\phi$ is an isomorphism. The following lemmas can be deduced from the work of Irmak (see \cite{Irmak}), but since we use them several times later, we present elementary and simple proofs. 
\begin{Lema}\label{pairofpants}
 Let $S_{1}$ and $S_{2}$ be infinite genus surfaces and let $\fun{\phi}{\comp{S_{1}}}{\comp{S_{2}}}$ be an isomorphism. If $\alpha$, $\beta$ and $\gamma$ are curves that bound a pair of pants on $S_{1}$, then their images bound a pair of pants on $S_{2}$.
\end{Lema}
\begin{proof}
If $\alpha \neq \beta = \gamma$, then $\beta$ cannot be an outer curve and hence its image is not an outer curve. Also, in any pants decomposition $P$ the curve $\beta$ will have degree one as a vertex of $\mathcal{A}(P)$. Hence $\phi(\beta)$ will also have degree one as vertex of $\mathcal{A}(\phi(P))$, given that (\ref{E:inducediso}) is an isomorphism. Then, the only option left is for $\beta$ to be the boundary of a pair of pants twice, as in the option of the left in figure \ref{deg1options}. Therefore $\phi(\alpha)$ and $\phi(\beta) = \phi(\gamma)$ bound a pair of pants on $S_{2}$.

\begin{figure}
\centering
\begin{tikzpicture}
\shadedraw[inner color=white!35!black, outer color=white!90!black, very thick] 
(-1+1,2-1) to[out=100,in=270] (-1.2+1,3-1) to[out=90,in=260]
(-1+1,4-1) to[out=180,in=0] (-4+1,4-1) to[out=190,in=90]
(-4.5+1,3-1) to[out=280,in=180] (-4+1,2-1) to[out=0,in=180] (-1+1,2-1);
\shadedraw[left color=white!40!black, right color= white, very thick] 
(-1+1,2-1) to[out=100,in=270] (-1.2+1,3-1) to[out=90,in=260]
(-1+1,4-1) to[out=280,in=90] (-.8+1,3-1) to[out=270,in=80]
(-1+1,2-1);
\filldraw[fill=white, very thick] 
(-3.25+1,2.8-1) to[out=340,in=180] (-2.75+1,2.6-1) to[out=0,in=200] (-2.25+1,2.8-1)
to[out=160,in=0] (-2.75+1,3-1) to[out=180,in=20] (-3.25+1,2.8-1);
\draw[ very thick] (-3.5+1,3-1) to[out=340,in=150](-3.25+1,2.8-1) to[out=340,in=180] (-2.75+1,2.6-1) to[out=0,in=200] (-2.25+1,2.8-1)
to[out=20,in=200] (-2+1,3-1);

\draw[very thick] (-3.25+1,2.8-1) to[out=110,in=0] 
(-3.8+1,3.5-1) node[anchor=south] {$v$}  to[out=180,in=70]  (-4.5+1,3-1);
 \draw [dashed, thick] (-4.5+1,3-1)  to[out=290,in=180] (-3.8+1,2.5-1)  to[out=0,in=250] (-3.25+1,2.8-1);
 
 
 \shadedraw[inner color=white!35!black, outer color=white!90!black, very thick] 
(2,0) to[out=90,in=270] (1,2) 
 to[out=290,in=180] (2.5-1,-.25+2)  to[out=0,in=240] (3-1,0+2)
(2,+2)  to[out=290,in=180] (2.5,-.25+2)  to[out=0,in=240] (3,0+2)  to[out=290,in=180] (2.5+1,-.25+2)  to[out=0,in=240] (3+1,0+2)
 to[out=270,in=90] (3,0) to[out=250,in=0] (2.5,-.25) to[out=180,in=290] (2,0);
 \shadedraw[inner color=white!35!black, outer color=white!90!black, very thick] 
(2+1,0+2) to[out=90,in=270] (1+1,2+2) 
 to[out=290,in=180] (2.5-1+1,-.25+2+2)  to[out=0,in=240] (3-1+1,0+2+2)
(2+1,+2+2)  to[out=290,in=180] (2.5+1,-.25+2+2)  to[out=0,in=240] (3+1,0+2+2)  to[out=290,in=180] (2.5+1+1,-.25+2+2)  to[out=0,in=240] (3+1+1,0+2+2)
 to[out=270,in=90] (3+1,0+2) to[out=250,in=0] (2.5+1,-.25+2) to[out=180,in=290] (2+1,0+2); 
 
 \draw[very thick] (2,0)  to[out=290,in=180] (2.5,-.25)  to[out=0,in=240] (3,0);
 \draw[thick, dashed]  (2,0)  to[out=70,in=180] (2.5,.25)  to[out=0,in=110] (3,0) node[anchor=west]{$u$};


 \shadedraw[top color=white, bottom color=white!40!black, very thick, shift={(1 ,2 )}] (2-1,0+2)  to[out=290,in=180] (2.5-1,-.25+2)  to[out=0,in=240] (3-1,0+2)
 to[out=110,in=0] (2.5-1,.25+2) 
  to[out=180,in=70] (2-1,0+2);

 \shadedraw[top color=white, bottom color=white!40!black, very thick, shift={(3 ,2 )}] (2-1,0+2)  to[out=290,in=180] (2.5-1,-.25+2)  to[out=0,in=240] (3-1,0+2)
 to[out=110,in=0] (2.5-1,.25+2) 
  to[out=180,in=70] (2-1,0+2);

 \shadedraw[top color=white, bottom color=white!40!black, very thick] (2-1,0+2)  to[out=290,in=180] (2.5-1,-.25+2)  to[out=0,in=240] (3-1,0+2)
 to[out=110,in=0] (2.5-1,.25+2) 
  to[out=180,in=70] (2-1,0+2);

 \draw[very thick] (2,+4)  to[out=290,in=180] (2.5,-.25+4)  to[out=0,in=240] (3,0+4);
 \draw[very thick, dashed]  (2,0+4)  to[out=70,in=180] (2.5,.25+4)  to[out=0,in=110] (3,0+4);


 \shadedraw[top color=white, bottom color=white!40!black, very thick] (2-1,0+2)  to[out=290,in=180] (2.5-1,-.25+2)  to[out=0,in=240] (3-1,0+2)
 to[out=110,in=0] (2.5-1,.25+2) 
  to[out=180,in=70] (2-1,0+2);
 
 
  \draw[very thick] (2+1,0+2)  to[out=290,in=180] (2.5+1,-.25+2)  to[out=0,in=240] (3+1,0+2)  node[anchor=west]{$v$} ;
 \draw[thick, dashed]  (2+1,0+2)  to[out=70,in=180] (2.5+1,.25+2)  to[out=0,in=110] (3+1,0+2);
\draw [very thick](-.5,1) to[out=120,in=240] (-.5,3);
\draw [dashed](-.5,1) to[out=70,in=290] (-.5,3);
\draw (-.5,3.2) node{$u$};

\end{tikzpicture}

\caption{The two options for $\deg(v)=1$.} \label{deg1options}
\end{figure}
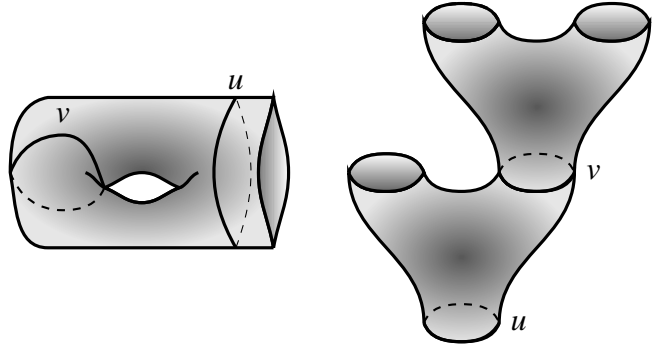

\indent It is impossible to bound a pair of pants using two separating curves and one nonseparating curve. Hence, if
$\alpha \neq \beta \neq \gamma \neq \alpha$, we only have the following cases according to the number of separating curves:
\begin{enumerate}
 \item Three separating curves. In this case, $\phi(\alpha)$, $\phi(\beta)$ and $\phi(\gamma)$ are three different separating curves, since separating curves go to separating curves as mentioned before. If these curves did not bound a pair of pants on $S_{2}$ we would have a pair of pants bounded by $\phi(\alpha)$ and $\phi(\beta)$ but not bounded by $\phi(\gamma)$, another pair of pants bounded by $\phi(\beta)$ and $\phi(\gamma)$ but not bounded by $\phi(\alpha)$ and another pair of pants bounded by $\phi(\gamma)$ and $\phi(\alpha)$ but not bounded by $\phi(\beta)$, as in figure \ref{pairofpantscirc}. But then none of these curves would be separating, leading us to a contradiction. Hence, $\phi(\alpha)$, $\phi(\beta)$ and $\phi(\gamma)$ bound a pair of pants on $S_{2}$.
 \item One separating curve. Let $\alpha$ and $\gamma$ be nonseparating curves and let $\beta$ be a separating curve. Then $\phi(\alpha)$ and $\phi(\gamma)$ are nonseparating curves and $\phi(\beta)$ is a separating curve, given the properties of $\phi$ mentioned before. If these curves did not bound a pair of pants on $S_{2}$, we would have a pair of pants bounded by $\phi(\alpha)$ and $\phi(\beta)$ but not bounded by $\phi(\gamma)$, another pair of pants bounded by $\phi(\beta)$ and $\phi(\gamma)$ but not bounded by $\phi(\alpha)$, but since $\phi(\beta)$ is a separating curve there cannot exist a pair of pants bounded by both $\phi(\alpha)$ and $\phi(\gamma)$, given that they are on different connected components of $S_{2} \backslash \{\phi(\beta)\}$, which leads us to a contradiction ($\phi(\alpha)$ and $\phi(\gamma)$ must be adjacent). Then $\phi(\alpha)$, $\phi(\beta)$ and $\phi(\gamma)$ bound a pair of pants on $S_{2}$.
 \item Three nonseparating curves. Given that $\alpha$, $\beta$ and $\gamma$ are nonseparating curves, we can always find a pants decomposition $P$ such that all their neighbours in $\adj{P}$ are nonseparating, $\alpha$ and $\beta$ have degree three in $\adj{P}$, $\gamma$ has degree four in $\adj{P}$ and $\alpha$ and $\gamma$ only have one common neighbour $\beta$ in $\adj{P}$.  For an example consider figure \ref{pairofpantsfig}. 
 Then, $\phi(\alpha)$ and $\phi(\beta)$ have degree three, $\phi(\gamma)$ has degree four, and all their neighbours are nonseparating. If $\phi(\alpha)$, $\phi(\beta)$ and $\phi(\gamma)$ do not bound a pair of pants on $S_{2}$ then there exist a pair of pants bounded by $\phi(\alpha)$, $\phi(\beta)$ and $\delta_{1} \neq \phi(\gamma)$, another pair of pants bounded by $\phi(\beta)$, $\phi(\gamma)$ and $\delta_{2} \neq \phi(\alpha)$, and another pair of pants bounded by $\phi(\alpha)$, $\phi(\gamma)$ and $\delta_{3} \neq \phi(\beta)$. Since $\phi(\beta)$ is the only common neighbour of $\phi(\alpha)$ and $\phi(\gamma)$, then $\delta_{3}$ is not an essential curve, which means it is isotopic to a boundary component, but this leads us to a contradiction, since $\phi(\gamma)$ would then have degree at most $3$.
 \end{enumerate}

\begin{figure}
\centering
\begin{tikzpicture}
\shadedraw[inner color=white!35!black, outer color=white!90!black, very thick]
(0,-1) to[out=0,in=100] (1.5,-2.5) to[out=80,in=180]
(2.5,-1.5) to[out=150,in=270] (1,2) to[out=190,in=0]
(0,1.7) to[out=180,in=350] (-1,2) to[out=270,in=30]
(-2.5,-1.5) to[out=0,in=100] (-1.5,-2.5) to[out=80,in=180] (0,-1);
\filldraw [fill=white, very thick] (0,0) circle (.5);

\shadedraw[top color=white!35!black, bottom color=white!90!black, very thick] (1.5,-2.5) to[out=80,in=180]
(2.5,-1.5)  to[out=260,in=10] (1.5,-2.5);

\shadedraw[top color=white!35!black, bottom color=white!90!black, very thick,rotate around={90:(2,-2)},shift={(0,3.99)}] (1.5,-2.5) to[out=80,in=180]
(2.5,-1.5)  to[out=260,in=10] (1.5,-2.5);

\shadedraw[top color=white!35!black, bottom color=white!90!black, very thick] (1,2) to[out=190,in=0]
(0,1.7) to[out=180,in=350] (-1,2) to[out=10,in=170] (1,2);


\draw[very thick] (0,-1) to[out=190,in=190] (0,-.5);
\draw[very thick, dashed] (0,-1) to[out=80,in=280] (0,-.5);

\draw[very thick] (0.5,0) to[out=280,in=190] (1.5,-.5);
\draw[very thick, dashed] (0.5,0) to[out=80,in=190] (1.5,-.5);

\draw[very thick] (-0.5,0) to[out=270,in=190] (-1.4,-.5);
\draw[very thick, dashed] (-0.5,0) to[out=80,in=190] (-1.4,-.5);

\draw (0,-1.5) node{$\phi(\gamma)$} ;
\draw (-2,0) node{$\phi(\alpha)$} ;
\draw (2,0) node{$\phi(\beta)$} ;

\draw (-1.7,-1.5) node{$\delta_{3}$} ;
\draw (1.6,-1.5) node{$\delta_{2}$} ;
\draw (0,1.2) node{$\delta_{1}$} ;

\end{tikzpicture}
\caption{If $\phi(\alpha)$, $\phi(\beta)$ and $\phi(\gamma)$ do not bound a pair of pants.}
\label{pairofpantscirc}
\end{figure}
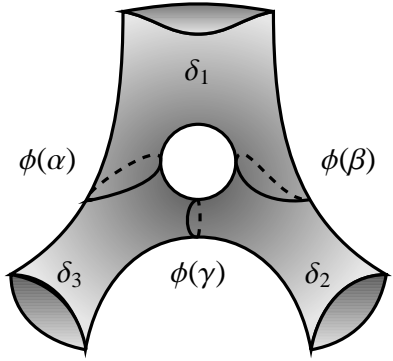

 
 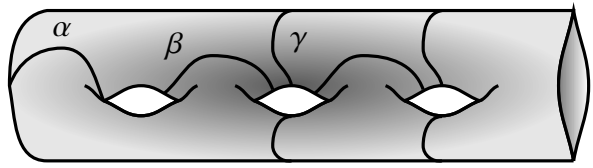
\begin{figure}
 \centering
 \begin{tikzpicture}
 
\shadedraw[inner color=white!35!black, outer color=white!90!black, very thick] 
(3,2) to[out=100,in=270] (2.8,3) to[out=90,in=260]
(3,4) to[out=180,in=0] (-4,4) to[out=190,in=90]
(-4.5,3) to[out=280,in=180] (-4,2) to[out=0,in=180] (3,2);

\shadedraw[left color=white!40!black, right color= white, very thick,shift={(3,1)}] 
(-1+1,2-1) to[out=100,in=270] (-1.2+1,3-1) to[out=90,in=260]
(-1+1,4-1) to[out=280,in=90] (-.8+1,3-1) to[out=270,in=80] (-1+1,2-1);

\foreach \x in {0,2,4}
\filldraw[fill=white, very thick, shift={(\x,0)}] 
(-3.25,2.8) to[out=340,in=180] (-2.75,2.6) to[out=0,in=200] (-2.25,2.8)
to[out=160,in=0] (-2.75,3) to[out=180,in=20] (-3.25,2.8);
\foreach \x in {0,2,4}
\draw[ very thick, shift={(\x,0)}] (-3.5,3) to[out=340,in=150](-3.25,2.8) to[out=340,in=180] (-2.75,2.6) to[out=0,in=200] (-2.25,2.8)
to[out=20,in=200] (-2,3);

\draw[very thick] (-3.25,2.8) to[out=110,in=0] 
(-3.8,3.5) node[anchor=south] {$\alpha$}  to[out=180,in=70]  (-4.5,3);
 
\foreach \x in {2,4}
\draw[very thick] (-2.75+\x,3) to[out=100,in=270] (-3+\x,3.5) to[out=90,in=190] (-2.75+\x,4);
\draw (-.65,3.3) node[anchor=south] {$\gamma$};
\foreach \x in {2,4}
\draw[very thick] (-2.75+\x,2.59) to[out=180,in=90] (-3+\x,2.2) to[out=270,in=190] (-2.75+\x,2); 
\foreach \x in {0,2}
\draw[very thick] (-2.44+\x,2.9) to[out=40,in=180] (-1.8+\x,3.4) to[out=0,in=110]
(-1+\x,2.95);
\draw (-2.3,3.2) node[anchor=south] {$\beta$};
 \end{tikzpicture}
\caption{Three nonseparating curves bounding a pair of pants.}
\label{pairofpantsfig}	
 \end{figure}

\end{proof}
\begin{remark}
 If $P$ is a pants decomposition and $\alpha \in P$ is a nonseparating curve of degree $2$ in $\adj{P}$ such that its neighbours are also nonseparating curves, then $\alpha$ forms part of two peripheral pairs, namely one with each neighbour (otherwise either of its neighbours or $\alpha$ itself would become separating).
 \end{remark}
\begin{Lema}\label{peripheralpairs}
 Let $S_{1}$ and $S_{2}$ be infinite genus surfaces and let $\fun{\phi}{\comp{S_{1}}}{\comp{S_{2}}}$ be an isomorphism. If $\alpha$ and $\beta$ form a peripheral pair, then their images form a peripheral pair. In particular, $S_{1}$ and $S_{2}$ have the same number of boundary components.
\end{Lema}
\begin{proof}
If $S_{1}$ admits at least $2$  peripheral pairs such that their curves are pairwise disjoint as in figure \ref{convpantsdec}, then we can always find a pants decomposition $P$ of $S_{1}$ such that all the neighbours of $\beta$ are nonseparating, $\deg(\alpha)=3$ and $\deg(\beta) = 2$. Then all the neighbours of $\phi(\beta)$ are nonseparating, and $\phi(\beta)$ has degree $2$, hence it has to form a peripheral pair with $\phi(\alpha)$ by the previous remark.\\
\indent If for any two peripheral pairs in $S_{1}$ at least one curve of each pair intersect each other, we can always find a pants decomposition $P$ of $S_{1}$ such that all the neighbours of $\alpha$ and $\beta$ are nonseparating, $\deg(\alpha)= \deg(\beta) = 3$, and there is only one pair of pants in $S_{1} \backslash P$ that is bounded by $\alpha$ and $\beta$ at the same time, namely the one formed by $\alpha$ and $\beta$ being a peripheral pair. Then $\phi(P)$ is a pants decomposition with all the neighbours of $\phi(\alpha)$ and $\phi(\beta)$ being nonseparating, $\deg(\phi(\alpha)) = \deg(\phi(\beta)) = 3$ and there exists a pair of pants in $S_{2}$ bounded by $\phi(\alpha)$, $\phi(\beta)$ and $\delta$. Due to lemma \ref{pairofpants} applied to $\phi$ and $\phi^{-1}$, $\delta$ cannot be an essential curve different to both $\phi(\alpha)$ and $\phi(\beta)$, but if $\delta = \phi(\alpha)$ or $\delta = \phi(\beta)$ then either $\phi(\beta)$ of $\phi(\alpha)$, respectively, becomes separating. Then $\delta$ is isotopic to a boundary component and so, $\phi(\alpha)$ and $\phi(\beta)$ form a peripheral pair.\\
\indent This result implies that $S_{2}$ has at least as many boundary components as $S_{1}$, and applying the same result to $\phi^{-1}$ we get that they have the same number of boundary components, even if this number is infinite.
\end{proof}
\begin{figure}
\centering
\begin{tikzpicture}

\shadedraw[inner color=white!35!black, outer color=white!90!black, very thick]

(-1,0) to[out=90,in=0] (-1.5,.5) to[out=45,in=335] (-1.5,1)
to[out=0,in=270] (-1,1.25) to[out=90,in=0] (-1.5,1.5)
to[out=0,in=270] (-1,2) to[out=0,in=180] (0,2) to[out=0,in=160] (1.5,1.5)
to[out=0,in=180] (4,1.5) to[out=220,in=140] (4,.5)
to[out=180,in=0] (1.5,.5) to[out=190,in=90] (1,0)
to[out=270,in=170] (1.5,-.5) to[out=0,in=180] (4,-.5)
to[out=220,in=140] (4,-1.5) to[out=180,in=0] (1.5,-1.5)
to[out=200,in=0] (0,-2) to[out=180,in=0] (-1,-2)
to[out=90,in=0] (-1.5,-1.5) to[out=0,in=270] (-1,-1.25)
to[out=90,in=0] (-1.5,-1) to[out=45,in=335] (-1.5,-.5)
to[out=0,in=270] (-1,0);


\draw[fill=white, very thick] (1.5,1) ellipse (.4 and .2);
\draw[fill=white, very thick] (3,1) ellipse (.4 and .2); 
\draw[fill=white, very thick] (1.5,-1) ellipse (.4 and .2);
\draw[fill=white, very thick] (3,-1) ellipse (.4 and .2); 
 
 \draw (-1,0) to[out=10,in=170] (-.5,0);
 \draw (-1,1.25) to[out=330,in=170] (-.25,.25);
 \draw (0,.5) to[out=70,in=290] (0,2);
 \draw (-1,-1.25) to[out=30,in=210] (-.25,-.25);
 \draw (0,-.5) to[out=290,in=70] (0,-2); 
 \draw (.25,.25) to[out=45,in=180] (1.1,1);
 \draw (.25,-.25) to[out=45,in=180] (1.1,-1);
 \foreach \x in {0,1.5}
 \foreach \y in {0,.7}
 \draw (1.5+\x,1.2-\y) to[out=45,in=325] (1.5+\x,1.5-\y);
 \foreach \x in {0,1.5}
\foreach \y in {0,.7}
\draw (1.5+\x,-1.2+\y) to[out=45,in=325] (1.5+\x,-1.5+\y);
\draw (0.5,0) to (.95,0);
 
 \shadedraw[left color=white!35!black, right color=white!90!black, very thick]  (4,1.5) to[out=220,in=140] (4,.5) to[out=50,in=320] (4,1.5);

 \shadedraw[left color=white!35!black, right color=white!90!black, very thick]  (4,1.5-2) to[out=220,in=140] (4,.5-2) to[out=50,in=320] (4,1.5-2); 
 
  \shadedraw[right color=white!35!black, left color=white!90!black, very thick] (-1.5,.5) to[out=45,in=335] (-1.5,1) to[out=205,in=135] (-1.5,.5);
 
  \shadedraw[right color=white!35!black, left color=white!90!black, very thick] (-1.5,.5-1.5) to[out=45,in=335] (-1.5,1-1.5) to[out=205,in=135] (-1.5,.5-1.5); 
 
\shadedraw[right color=white!35!black, left color=white!90!black, very thick] (-1.5,1.5)
to[out=0,in=270] (-1,2) to[out=220,in=90] (-1.5,1.5) ;

\shadedraw[right color=white!35!black, left color=white!90!black, very thick] (-1,-2)
to[out=90,in=0] (-1.5,-1.5) to[out=270,in=130] (-1,-2) ;
 
\draw [fill=white,very thick] (0,0) circle (0.5);
 
 \draw (-.5,1) node{$\beta$};
 \draw (.33,1.5) node{$\alpha$};
  \foreach \x in {.2,.5,.8}
 \filldraw[fill=black] (4.3+\x,1) circle (0.05);
  \foreach \x in {.2,.5,.8}
 \filldraw[fill=black] (4.3+\x,1-2) circle (0.05);
 
\end{tikzpicture}
\caption{Example of a convenient pants decomposition.}
\label{convpantsdec}
\end{figure}
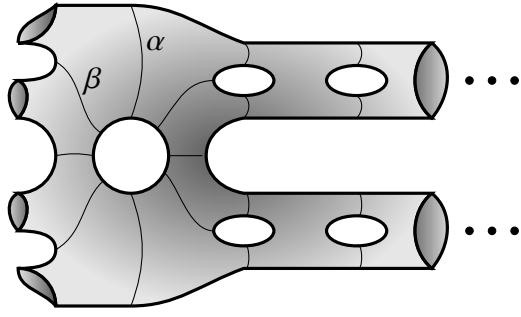

\textbf{Proof of theorem \ref{S1homeoS2}}. Let $P$ be a pants decomposition of $S_{1}$. 
From the fact that (\ref{E:inducediso}) is an isomorphism and theorem \ref{GraphEndsSurfaceEnds} we have $\Ends{S_{1}} \cong \Ends{\adj{P}} \cong \Ends{\adj{\phi(P)}} \cong \Ends{S_{2}}$. From the surface classification theorem for infinite surfaces by Richards, Prishlyak and Mischenko in \cite{Ray} and \cite{PM}, it is sufficient to prove that $S_{1}$ and $S_{2}$ have the same number of boundary components to guarantee that they are homeomorphic. But this is guaranteed by lemma \ref{peripheralpairs}.\qed 
\begin{remark}
 Theorem  \ref{S1homeoS2} cannot be extended for infinite genus surfaces with punctures. Indeed, let $S$ be an infinite genus surface with $n>0$ boundary components and without planar ends. Let $S'$ be the infinite genus surface obtained from $S$ by glueing one punctured disc to $S$ along a boundary component. Clearly $S$ and $S'$ are not homeomorphic, but $\comp{S}\cong\comp{S'}$.
\end{remark}


\textbf{Proof of corollary \ref{CorIsoNSC}}. The statement is immediate for all arguments given in the proof of theorem 2 remain valid  if we change $\comp{S}$ for $\nons{S}$ and take all pants decompositions to be formed just by nonseparating curves. \qed
\subsection{Surjectivity.}
	\label{NaturalMap}
	At the end of this section we give a proof for theorem \ref{TeoISO}. 
We begin by proving the following theorems: 
\begin{Teo}
		\label{Teo6}
Let $S$ be an infinite genus surface. Then $\Aut{\gs{S}} \cong \Aut{\nons{S}}$. 
\end{Teo}
\begin{Teo}\label{surGS}
Let $S$ be an infinite genus surface such that
$\Ends{S} = \Endsg{S}$. The natural map: 
\begin{equation}
\Psi_{\gs{S}}:{\rm MCG}^{*}(S)\to {\rm Aut}(\gs{S})
\end{equation}
is surjective. 
\end{Teo}
This two results imply that the natural map:
\begin{equation}
\Psi_{\nons{S}}:{\rm MCG}^{*}(S)\to {\rm Aut}(\mathcal{N}(S))
\end{equation}
is surjective. Using the surjectivity of this map, we can deduce the following:
\begin{Teo}\label{surCS}
Let $S$ be an infinite genus surface such that
$\Ends{S} = \Endsg{S}$. The natural map: 
\begin{equation}
\Psi_{\comp{S}}:{\rm MCG}^{*}(S)\to {\rm Aut}(\comp{S})
\end{equation}
is surjective.
\end{Teo}

\subsubsection{Proof of theorems \ref{Teo6} and  \ref{surGS}.}
The proofs of theorems \ref{Teo6} and  \ref{surGS} require some auxiliary lemmas given in \cite{Irmak} and \cite{Sch} but adapted to the context of infinite type surfaces. When the proofs of these lemmas can be easily deduced from the cited works we just state them without a proof. When this is not the case elementary and simple proofs are provided. We recall first the different components that a curve might have. 
\begin{Def}[\sc Curve components]
 Let $\alpha$ and $\beta$ be nonseparating curves such that $i(\alpha,\beta) \geq 2$. Let $\beta_{1}$ be a connected component of $\beta$ in $S_{\alpha}$. If the surface resulting from cutting $S_{\alpha}$ along $\beta_{1}$ is connected, then $\beta_{1}$ is called a nonseparating component of $\beta$ (with respect to $\alpha$). Otherwise, $\beta_{1}$ is called a separating component of $\beta$ (with respect ot $\alpha$). If $\beta_{1}$ connects the two different boundary components of $S_{\alpha}$ induced by $\alpha$, then $\beta_{1}$ is called a two-sided component. Otherwise it is called one-sided.
\end{Def}

\begin{Lema}\label{Schmutz3}\cite{Sch}
 Let $S$ be an infinite genus surface and $\alpha, \beta \in \mathcal{V}(\nons{S})$ such that $i(\alpha,\beta) \geq 2$. If $\beta$ has a nonseparating component $\beta_{1}$ with respect to $\alpha$, then there exists $\gamma, \gamma^{\prime} \in \mathcal{V}(\nons{S})\backslash\{\alpha,\beta\}$ such that $N(\alpha,\beta) \subset (N(\gamma) \cup N(\gamma^{\prime}))$. Moreover, if $\beta_{1}$ is one-sided, then $\alpha, \gamma, \gamma^{\prime}$ are mutually disjoint; if $\beta_{1}$ is two-sided, then $\{\alpha,\gamma, \gamma^{\prime}\}$ is a triple with $$i(\alpha,\beta) = i(\beta,\gamma) + i(\beta,\gamma^{\prime})$$ and $\min\{i(\beta,\gamma), i(\beta,\gamma^{\prime})\} > 0$.
\end{Lema}

\begin{Lema}\label{Int0}\emph{[\emph{Ibid.}]}
 Let $S_{1}$ and $S_{2}$ be infinite genus surfaces and let $\fun{\phi}{\gs{S_{1}}}{\gs{S_{2}}}$ be an isomorphism. Then for any disjoint curves $\alpha$ and $\beta$, their images under $\phi$ will also be disjoint.
\end{Lema}

\textbf{Proof theorem \ref{Teo6}}.
 Let $\phi \in \Aut{\nons{S}}$. Since any automorphism of $\nons{S}$ (and $\gs{S}$ respectively) is uniquely determined by the function on its vertices and $\mathcal{V}(\nons{S}) = \mathcal{V}(\gs{S})$, then $\phi$ induces a bijection $\fun{\phi^{*}}{\gs{S}}{\gs{S}}$. From the work of Irmak \cite{Irmak} on the characterization of two curves that intersect once, one can deduce that if $S_{1}$ and $S_{2}$ are infinite genus surfaces, and $\fun{\phi_{1}}{\comp{S_{1}}}{\comp{S_{2}}}$ and $\fun{\phi_{2}}{\nons{S_{1}}}{\nons{S_{2}}}$ are isomorphisms, then for any curves $\alpha_{1}$ and $\alpha_{2}$ such that $i(\alpha_{1},\alpha_{2})=1$ we have that $i(\phi_{1}(\alpha_{1}),\phi_{1}(\alpha_{2})) = i(\phi_{2}(\alpha_{1}),\phi_{2}(\alpha_{2})) = 1$. This fact applied to $\phi$ and $\phi^{-1}$ implies that $\phi^{*}$ must preserve adjacency and non-adjacency. Hence we can define the function
\begin{equation}
    \label{E:fiso}
\fun{\Phi}{\Aut{\nons{S}}}{\Aut{\gs{S}}}
\end{equation} 
 as $\phi \mapsto \phi^{*}$. This function is clearly an injective group homomorphism.\\
 \indent In the same way, for any automorphism of $\gs{S}$ we can induce a bijection from $\nons{S}$ to itself, and due to lemma \ref{Int0} this bijection will become an automorphism of $\nons{S}$. Therefore $\Phi$ is an isomorphism.
\qed
\begin{remark}
	\label{R:changing}
From the proof of theorem  \ref{Teo6} and the proof of corollary \ref{CorIsoNSC} we conclude that the statements of lemmas \ref{pairofpants} and \ref{peripheralpairs} remain valid if we change $\comp{S}$ for $\gs{S}$.
\end{remark}
The following four lemmas are used in the proof of theorem  \ref{surGS}. Let us recall first the notion of triple of curves. 

\begin{Def}[\sc Triples of curves]
 Let $\alpha$, $\beta$ and $\gamma$ be nonseparating curves of $S$. We will say $\{\alpha,\beta,\gamma\}$ is a triple if $i(\alpha,\beta) = i(\alpha,\gamma) = i(\beta,\gamma) = 1$ and there exists a subsurface of $S$ which contains $\alpha$, $\beta$ and $\gamma$, and is homeomorphic to a torus with one boundary component.
\end{Def}
%

\begin{Lema}\label{Schmutz3prime}
 Let $S$ be infinite genus surface and $\alpha, \beta \in \mathcal{V}(\nons{S})$ be such that $i(\alpha,\beta) \geq 2$. If $\beta$ does not have two-sided components with respect to $\alpha$, then there exists $\gamma, \gamma^{\prime} \in \mathcal{V}(\nons{S})\backslash\{\alpha,\beta\}$ such that $\{\alpha,\gamma, \gamma^{\prime}\}$ is a triple with $$i(\alpha,\beta) = i(\beta,\gamma) + i(\beta,\gamma^{\prime})$$ and $\min\{i(\beta,\gamma), i(\beta,\gamma^{\prime})\} > 0$.
\end{Lema}
\begin{proof}
 Let $\alpha_{1}$ and $\alpha_{2}$ be the boundary components on $S_{\alpha}$ induced by $\alpha$. Since $\beta$ does not have two-sided components then it only has one-sided components and therefore we can choose a curve $\gamma$ that intersects $\alpha$ once, does not intersect any one sided component of $\beta$ based on $\alpha_{1}$ and intersects $\beta$ in such a way that $0 < i(\gamma,\beta) \leq \frac{1}{2}i(\alpha,\beta)$. This can be done by drawing $\gamma$ disjoint from every one-sided component of $\beta$ based on $\alpha_{1}$, we keep on going ``following'' a convenient one-sided component of $\beta$ based on $\alpha_{2}$ until before we reach $\alpha_{2}$, then we intersect $\alpha_{2}$ in the corresponding point seeking the desired inequality. See figure \ref{3primefig1} for examples.
 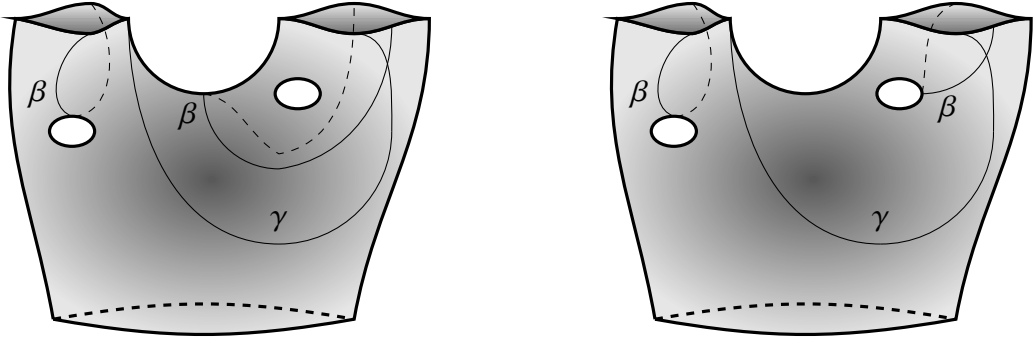
\begin{figure}
 \centering
 	\begin{tikzpicture}

\shadedraw[inner color=white!35!black, outer color=white!90!black, very thick]

(-2,-1) to[out=350,in=190] (2,-1)
to[out=80,in=270] (3,3)
to[out=180,in=0] (2,2.8)
to[out=180,in=0] (1,3)
to[out=270,in=0] (0,2)
to[out=180,in=270] (-1,3)
to[out=190,in=10] (-1.5,2.8)
to[out=180,in=350] (-2.5,3)
to[out=260,in=100] (-2,-1);

\shadedraw[bottom color=white!35!black, top color=white!90!black, very thick]
(3,3)
to[out=180,in=0] (2,2.8)
to[out=180,in=0] (1,3)
to[out=0,in=180] (2,3.2)
to[out=0,in=180] (3,3);
\shadedraw[bottom color=white!35!black, top color=white!90!black, very thick]
(-1,3)
to[out=190,in=10] (-1.5,2.8)
to[out=180,in=350] (-2.5,3)
to[out=0,in=180] (-1.5,3.2)
to[out=0,in=180] (-1,3);


\filldraw [fill=white, very thick] (-1.75,1.5) ellipse (.3 and .2);
\filldraw [fill=white, very thick] (1.25,2) ellipse (.3 and .2);


\draw (-1.5,2.8) to[out=200,in=160] (-1.75,1.7);
\draw[dashed] (-1.5,3.2) to[out=310,in=0] (-1.75,1.7);
\draw (0,2) to[out=270,in=180] (1,1)
to[out=8,in=270] (2.5,2.9);
\draw [dashed] (0,2) to[out=350,in=150] (1,1.2)
to[out=10,in=270] (2,3.2)
;
\draw (-1,3) to[out=270,in=180] (1,0);
\draw (1,0) to[out=0,in=270] (2.5,1.5);
\draw (2.5,1.5) to[out=90,in=0] (2,2.8)
;
\draw[dashed,very thick] (-2,-1) to[out=10,in=170] (2,-1);

\draw (-2.2,2) node{$\beta$};
\draw (1,.3) node{$\gamma$};
\draw (-.2,1.7) node{$\beta$};  

\begin{scope}[xshift=8cm]
\shadedraw[inner color=white!35!black, outer color=white!90!black, very thick]

(-2,-1) to[out=350,in=190] (2,-1)
to[out=80,in=270] (3,3)
to[out=180,in=0] (2,2.8)
to[out=180,in=0] (1,3)
to[out=270,in=0] (0,2)
to[out=180,in=270] (-1,3)
to[out=190,in=10] (-1.5,2.8)
to[out=180,in=350] (-2.5,3)
to[out=260,in=100] (-2,-1);

\shadedraw[bottom color=white!35!black, top color=white!90!black, very thick]
(3,3)
to[out=180,in=0] (2,2.8)
to[out=180,in=0] (1,3)
to[out=0,in=180] (2,3.2)
to[out=0,in=180] (3,3);
\shadedraw[bottom color=white!35!black, top color=white!90!black, very thick]
(-1,3)
to[out=190,in=10] (-1.5,2.8)
to[out=180,in=350] (-2.5,3)
to[out=0,in=180] (-1.5,3.2)
to[out=0,in=180] (-1,3);


\filldraw [fill=white, very thick] (-1.75,1.5) ellipse (.3 and .2);
\filldraw [fill=white, very thick] (1.25,2) ellipse (.3 and .2);


\draw (-1.5,2.8) to[out=200,in=160] (-1.75,1.7);
\draw[dashed] (-1.5,3.2) to[out=310,in=0] (-1.75,1.7);
\draw (1.55,2) to[out=0,in=270] (2.5,2.9);
\draw [dashed] (1.55,2) to[out=80,in=200] (2,3.2);
\draw (-1,3) to[out=270,in=180] (1,0);
\draw (1,0) to[out=0,in=270] (2.5,1.5);
\draw (2.5,1.5) to[out=90,in=0] (2,2.8)
;
\draw[dashed,very thick] (-2,-1) to[out=10,in=170] (2,-1);

\draw (-2.2,2) node{$\beta$};
\draw (1,.3) node{$\gamma$};
\draw (1.9,1.8) node{$\beta$};

\end{scope}
\end{tikzpicture}

\caption{Examples of $\beta$ and $\gamma$ in $S_{\alpha}$}
\label{3primefig1}
 \end{figure}

\indent Then let $N$ be a regular neighbourhood of $\alpha$ and $\gamma$; since $i(\alpha,\gamma) = 1$ then $N$ is homeomorphic to a torus with one boundary component. Let $\gamma^{\prime}$ be the image of $\gamma$ under a Dehn twist along $\alpha$ on $N$. See figure \ref{3primefig2} for the corresponding diagram.
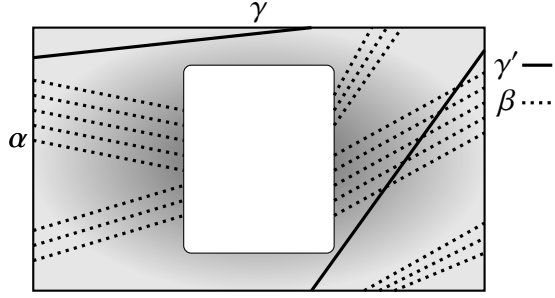
\begin{figure}
\centering
\begin{tikzpicture}
\shadedraw[inner color=white!35!black, outer color=white!90!black,thick]
(0,0)--(6,0)--(6,3.5)--(0,3.5)--cycle;
\filldraw [rounded corners=1mm,fill=white] (2,.5)--(4,.5)--(4,3)--(2,3)--cycle;


\foreach \x in{0,.2,.4}
\draw[dotted, very thick] (0,.4+\x)--(2,1+\x);

\foreach \x in{0,.2,.4,.6,.8}
\draw[dotted, very thick] (0,2+\x)--(2,1.6+\x);

\draw[very thick] (0,3.1)--(3.7,3.5);


\foreach \x in{0,.2,.4,.6,.8}
\draw[dotted, very thick] (4,1+\x)--(6,2.1+\x);

\foreach \x in{0,.2,.4}
\draw[dotted, very thick] (4.8-\x,0)--(6,.5+\x);

\foreach \x in{0,.2,.4}
\draw[dotted, very thick] (4,2.2+\x)--(4.9-\x,3.5);

\draw[very thick] (3.7,0)--(6,3.2);


\draw (-.2,2) node{$\alpha$};
\draw (3,3.7) node{$\gamma$};
\draw (-.2,2) node{$\alpha$};
\draw (6.3,3) node{$\gamma'$};
\draw[very thick] (6.5,3)--(6.9,3) ;
\draw (6.3,2.5) node{$\beta$};
\draw[dotted, very thick] (6.5,2.5)--(6.9,2.5) ;
\end{tikzpicture}
\caption{Diagram of $N$.}
\label{3primefig2}
\end{figure}

\indent Thus $\{\alpha,\gamma,\gamma^{\prime}\}$ form a triple and by construction $i(\beta,\gamma^{\prime}) = i(\alpha,\beta) - i(\beta,\gamma)$, with both curves intersecting $\beta$ at least once.
\end{proof}

\begin{Lema}\label{Intn}
 Let $S$ be an infinite genus surface and let $\fun{\phi}{\gs{S}}{\gs{S}}$ be an automorphism. Then $i(\alpha,\beta) = i(\phi(\alpha),\phi(\beta))$ for all $\alpha,\beta \in \mathcal{V}(\gs{S})$.
\end{Lema}
\begin{proof}
 Let $\alpha,\beta \in \mathcal{V}(\gs{S})$. If $i(\alpha,\beta) = 0$, then due to lemma \ref{Int0} we have that $i(\phi(\alpha),\phi(\beta)) = 0$. If $i(\alpha,\beta) = 1$, then due to $\phi$ being an automorphism $i(\phi(\alpha),\phi(\beta)) = 1$. For $i(\alpha,\beta) \geq 2$, we will proceed by induction on the geometric intersection number.\\
 \indent Let us suppose the geometric intersection number is preserved under automorphisms for curves which intersect at most $k$ times for a $k \geq 1$.\\
 \indent Let $i(\alpha,\beta) = k+1$. Due to lemmas \ref{Schmutz3} and \ref{Schmutz3prime}, we know there exists $\gamma,\gamma^{\prime} \in \mathcal{V}(\gs{S}) \backslash \{\alpha,\beta\}$ such that $\{\alpha,\gamma,\gamma^{\prime}\}$ is a triple, $i(\alpha,\beta) = i(\beta,\gamma) + i(\beta,\gamma^{\prime})$ and $\min\{i(\beta,\gamma),i(\beta,\gamma^{\prime})\} > 0$.\\
 \indent Since $i(\beta,\gamma), i(\beta,\gamma^{\prime}) < k+1$, then $i(\beta,\gamma) = i(\phi(\beta),\phi(\gamma))$ and $i(\beta,\gamma^{\prime}) = i(\phi(\beta),\phi(\gamma^{\prime}))$. From the work of Schmutz [\emph{Ibid}.] one can deduce that if $S$ is an infinite genus surface and $\fun{\phi}{\gs{S}}{\gs{S}}$ an automorphism, then for every triple $\{\alpha,\beta,\gamma\}$ we have that $\{\phi(\alpha),\phi(\beta),\phi(\gamma)\}$ form a triple.
 Therefore $\{\phi(\alpha),\phi(\gamma),\phi(\gamma^{\prime})\}$ form a triple. Using a diagram of the torus with one boundary component which contains this triple (see figure ~\ref{3primefig2}), we can see that each time $\phi(\beta)$ intersects $\phi(\alpha)$ then either $\phi(\beta)$ intersects $\phi(\gamma)$ or $\phi(\beta)$ intersects $\phi(\gamma')$. Therefore $i(\phi(\beta), \phi(\gamma)) + i(\phi(\beta),\phi(\gamma^{\prime})) \geq i(\phi(\alpha),\phi(\beta))$. Thus $i(\alpha,\beta) \geq i(\phi(\alpha),\phi(\beta))$. Applying the same argument on $\phi^{-1}$ we obtained the symmetric inequality, therefore $i(\alpha,\beta) = i(\phi(\alpha),\phi(\beta))$.
\end{proof}
\begin{Lema}\label{homeopantsdec}
 Let $S$ be an infinite genus surface and let $\fun{\phi}{\gs{S}}{\gs{S}}$ be an automorphism. If $P$ is a pants decomposition of $S$, then there exist an homeomorphism $h \in \MCG{S}$ such that $h(\alpha) = \phi(\alpha)$ for all $\alpha \in P$.
\end{Lema}
\begin{proof}
 From remark \ref{R:changing}, we know that $\phi(P)$ is a pants decomposition and the boundaries of pair of pants in $S_{\phi(P)}$ induced by curves of $\phi(P)$ are boundaries of pair of pants in $S_{P}$ induced by curves of $P$. Then we can define an homeomorphism of $S$ by parts using homeomorphisms from the connected components of $S_{P}$ to the corresponding connected components of $S_{\phi(P)}$; this homeomorphism by construction will agree with $\phi$ for every element in $P$.
\end{proof}
\begin{remark}
 It is clear, using theorem \ref{Teo6}, that this lemma remains valid if we substitute $\gs{S}$ by $\nons{S}$.
\end{remark}
\begin{Lema}\label{Sch04}\cite{Sch}
 Let $S^{\prime}$ be a surface of genus zero and four boundary components. Let $\alpha,\beta \in \mathcal{V}(\comp{S^{\prime}})$ with $i(\alpha,\beta) =2$.
 \begin{enumerate}
  \item Let $\gamma \in \mathcal{V}(\comp{S^{\prime}})$ such that $i(\alpha,\gamma) = 2$. Then there exists $h\in \MCG{S^{\prime}}$ such that $h(\alpha) = \alpha$ and $h(\beta) = \gamma$.
  \item There are exactly two curves $\gamma_{1},\gamma_{2} \in \mathcal{V}(\comp{S^{\prime}})$ such that $i(\alpha, \gamma_{i}) = i(\beta,\gamma_{i}) = 2$ for $i=1,2$. Moreover, there exists $h \in \MCG{S^{\prime}}$ such that $h(\alpha) = \alpha$, $h(\beta) = \beta$, and $h(\gamma_{1}) = \gamma_{2}$.
 \end{enumerate}
\end{Lema}
\begin{remark}
The homeomorphism of part (1) in the preceding lemma is just a Dehn twist about $\alpha$, where as the homeomorphism from part (2) is an orientation-reversing involution that leaves invariant each connected component in the boundary of $S_{0,4}$.
\end{remark}
 The proof theorem \ref{surGS} uses the notion of Dehn-Thurston coordinates. Therefore we recall it and discuss it briefly in the context of infinite surfaces in the following paragraphs.

\begin{Def}[\sc Dehn-Thurston coordinates]
 A Dehn-Thurston coordinates system of curves is a set $D$ of curves that parametrize every curve $\alpha \in \mathcal{V}(\comp{S})$ using the geometric intersection number, i.e. for $\alpha,\beta \in \mathcal{V}(\comp{S})$ if $i(\alpha,\gamma) = i(\beta,\gamma)$ for all $\gamma \in D$, then $\alpha = \beta$.
\end{Def}
\indent For compact surface, it is well known that Dehn-Thurston coordinate systems exist, see \cite{HP}. For noncompact surfaces such a system of curves can be realized in the following way. Let $\{\alpha_{i}\}_{i\in \mathbf{N}}$ be a pants decomposition, $\{\beta_{i}\}_{i \in \mathbf{N}}$ be curves such that $i(\alpha_{i},\beta_{i})=2$ and $i(\alpha_{i},\beta_{j}) =0$ for $i \neq j$, and $\{\gamma_{i}\}_{i\in \mathbf{N}}$ be curves such that $i(\alpha_{i},\gamma_{i}) = i(\beta_{i},\gamma_{i}) = 2$ and $i(\alpha_{i},\gamma_{j}) = 0$ for $i \neq j$. Then the set of curves $D$ formed by the union of elements in  $\{\alpha_{i}\}_{i \in \mathbf{N}}$, $\{\beta_{i}\}_{i \in \mathbf{N}}$ and $\{\gamma_{i}\}_{i \in \mathbf{N}}$ is a Dehn-Thurston coordinate system. Indeed, any curve $\delta$ in $S$ will only intersect finitely many curves in $D$, hence we can take any compact subsurface $S^{\prime}$, such that it contains $\delta$ and there is a (finite) subset $D^{\prime}$ of $D$ that is a Dehn-Thurston coordinate system of $S^{\prime}$.  Any other curve in $S$ with the same Dehn-Thurston coordinates as $\delta$ on the system $D$, would have to be isotopic to a curve contained in $S^{\prime}$ and thus would have the same Dehn-Thurston coordinates as $\delta$ on the system $D^{\prime}$, therefore it would be isotopic to $\delta$. We must remark that, when $S$ is an infinite genus surface such that $\Ends{S} = \Endsg{S}$, we can also construct the Dehn-Thurston coordinate system $D$ with families  $\{\alpha_{i}\}_{i \in \mathbf{N}}$, $\{\beta_{i}\}_{i \in \mathbf{N}}$ and $\{\gamma_{i}\}_{i \in \mathbf{N}}$ formed exclusively by nonseparating curves. \\

\textbf{Proof theorem \ref{surGS}}. Given that $\Ends{S} = \Endsg{S}$ we can construct $P = \{\alpha_{i}\}_{i \in \mathbf{N}}$ a pants decomposition of $S$ formed by nonseparating curves.  Let $\phi:\gs{S}\to\gs{S}$ an automorphism. Due to lemma \ref{homeopantsdec} there exists an homeomorphism $\fun{h_{1}}{S}{S}$ such that $h_{1}(\alpha_{i}) = \phi(\alpha_{i})$ for all $\alpha_{i} \in P$.\\
\indent Again, since $\Ends{S} = \Endsg{S}$ we can construct $\{\beta_{i}\}_{i \in \mathbf{N}}$ a collection of nonseparating curves such that $i(\alpha_{i},\beta_{i}) = 2$ for all $i$ and $i(\alpha_{i},\beta_{j}) = 0$ for $i \neq j$. We can define an homeomorphism $\fun{h_{2}}{S}{S}$ such that $h_{2}(h_{1}(\alpha_{i})) = h_{1}(\alpha_{i}) = \phi(\alpha_{i})$ and $h_{2}(h_{1}(\beta_{i})) = \phi(\beta_{i})$ in the following way. For every $i \in\mathbf{N}$ the curves $\alpha=h_{1}(\alpha_{i})$, $\beta=h_{1}(\beta_{i})$ and $\gamma=\phi(\beta_{i})$ satisfy the hypotheses of part (1) in lemma \ref{Sch04} and lie in a subsurface $S_{i}$ homeomorphic to $S_{0,4}$ that does not contain any element in $h_{1}(P) \backslash \{h(\alpha_{i})\}$, $S_{i}$ contains $h_{1}(\beta_{i})$ and $\phi(\beta_{i})$, and its boundary components are isotopic to the curves adjacent to $h_{1}(\alpha_{i})$ with respect to $h_{1}(P)$. Let $\fun{h_{2,i}}{S_{i}}{S_{i}}$ be the homeomorphism from $(1)$ in lemma \ref{Sch04}. This homeomorphism is just a Dehn twist about $\alpha$, therefore it preserves orientation  and its support $K_{i}\subset S_{i}$ satisfies that $K_{i}\cap K_{j}=\emptyset$ for $i\neq j$ for all $i,j\in\mathbf{N}$. Hence $h_{2}$ can be defined by parts using $\{h_{2,i}\}_{i \in \mathbf{N}}$.\\
\indent Let $\{\gamma_{i}\}_{i \in \mathbf{N}}$ be a collection of curves such that $i(\alpha_{i},\gamma_{i}) = i(\beta_{i},\gamma_{i}) = 2$ and $i(\alpha_{i},\gamma_{j}) = 0$ for $i \neq j$. We can define an homeomorphism $\fun{h_{3}}{S}{S}$ such that $h_{3}(h_{2}(h_{1}(\alpha_{i}))) = h_{2}(h_{1}(\alpha_{i}) = \phi(\alpha_{i})$, $h_{3}(h_{2}(h_{1}(\beta_{i}))) = h_{2}(h_{1}(\beta_{i})) = \phi(\beta_{i})$ and $h_{3}(h_{2}(h_{1}(\gamma_{i}))) = \phi(\gamma_{i})$ in the following way. For every $i\in\mathbf{N}$, let now $\alpha=h_{1}(h_{2}(\alpha_{i}))$, $\beta=h_{1}(h_{2}(\beta_{i}))$, $\gamma_{1}=h_{1}(h_{2}(\gamma_{i}))$ and $\gamma_{2}=\phi(\gamma_{i})$. Analogously to the preceding case, these curves satisfy the hypotheses of part (2) in lemma \ref{Sch04}. Let $\fun{h_{3,i}}{R_{i}}{R_{i}}$ be the (orientation-reversing) homeomorphism from part (2) in lemma \ref{Sch04}, where $R_{i}$ is homeomorphic to $S_{0,4}$ and contains the curves $\alpha$, $\beta$, $\gamma_{1}$ and $\gamma_{2}$.  It is not difficult to see that if $i\neq j$ and $R_{i}\cap R_{j}\neq \emptyset$, then $R_{i,j}=R_{i}\cap R_{j}\cong S_{0,3}$. Moreover $h_{3,i}$ and $h_{3,j}$ coincide in $R_{i,j}$ and hence we can define $h_{3}$ by parts using $\{h_{3,i}\}_{i\in\mathbf{N}}$.\\
\indent Let $h = h_{3} \circ h_{2} \circ h_{1}$. Since $P^{\prime} = P \cup \{\beta_{i}\} \cup \{\gamma_{i}\}$ form a Dehn-Thurston coordinates system of curves, then $h(P^{\prime})$ is a Dehn-Thurston coordinates system of curves, and by construction $h(\varepsilon) = \phi(\varepsilon)$ for all $\varepsilon \in P^{\prime}$. Therefore, due to lemma \ref{Intn}, for all $\delta \in \mathcal{V}(\gs{S})$ and all $\varepsilon \in P^{\prime}$:
\begin{equation}
i(\phi(\delta),\phi(\varepsilon)) = i(\delta,\varepsilon) = i(h(\delta),h(\varepsilon)) = i(h(\delta),\phi(\varepsilon)),
\end{equation}
then $\phi(\delta) = h(\delta)$ for all $\delta \in \mathcal{V}(\gs{S})$, which implies $\Psi_{\gs{S}}$ is surjective. \qed

\begin{Cor}
	\label{surNS} 
 Let $S_{1}$ and $S_{2}$ be infinite genus surfaces, such that $\Ends{S_{i}} = \Endsg{S_{i}}$ for $i = 1,2$ and let $\fun{\phi}{\gs{S_{1}}}{\gs{S_{2}}}$ be an isomorphism. Then $S_{1}$ and $S_{2}$ are homeomorphic and $\phi$ is induced by a mapping class in ${\rm MCG}^{*}(S_{1})$.
\end{Cor}
\begin{proof}
Every isomorphism $\fun{\phi}{\gs{S_{1}}}{\gs{S_{2}}}$ induces an isomorphism $\fun{\phi}{\mathcal{N}(S_{1})}{\mathcal{N}(S_{2})}$.  Indeed, take $u,v$ two curves such that $i(u,v)=0$. Suppose $i(\phi(u),\phi(v))\geq 2$ and remark that, as in the proof of theorem \ref{surGS}, lemmas 1, 3 and 5 in \cite{Sch} remain valid in the context of this corollary. Hence we obtain a contradiction. On the other hand it is clear that $i(\phi(u),\phi(v))\neq 1$, for $\phi$ is an isomorphism. Hence the only possibility left is that $i(\phi(u),\phi(v))=0$. By corollary \ref{CorIsoNSC} we obtain that $S_{1}$ is homemorphic to $S_{2}$. The rest of the proof follows from theorems \ref{TeoISO} and \ref{surGS}.
\end{proof}

\subsubsection{Proof of theorem \ref{surCS}.}
\label{pt8}
Any $\phi \in \Aut{\comp{S}}$ sends nonseparating curves to nonseparating curves, hence $\phi|_{\nons{S}} \in \Aut{\nons{S}}$ and then due to theorem \ref{surGS} there exists $h \in \MCG{S}$ such that $\phi|_{\nons{S}}(\alpha) = h(\alpha)$ for all $\alpha \in \mathcal{V}(\nons{S})$. Hence we only need to check that $\phi$ and $h$ coincide in the separating curves of $S$. Let $\alpha$ be a separating curve of $S$; we consider three cases.\\
\begin{enumerate}
 \item If both connected components of $S_{\alpha}$ have positive genus, then we can find a pants decomposition $P$ such that $\alpha \in P$, $(P \backslash \{\alpha\}) \subset \mathcal{V}(\nons{S})$ and $\deg(\alpha) = 4$ in $\adj{P}$; let $\beta_{1}$, $\gamma_{1}$, $\beta_{2}$ and $\gamma_{2}$ be the neighbours of $\alpha$ in $\adj{P}$ such that $\beta_{i}$ and $\gamma_{i}$ are in the same connected component of $S_{\alpha}$ for $i = 1,2$. Let also $\delta_{1}$ and $\delta_{2}$ be nonseparating curves such that $i(\alpha, \delta_{i}) = 0$ and $i(\beta_{i},\delta_{i}) = i(\gamma_{i},\delta_{i}) = 1$ for $i = 1,2$. See figure \ref{Catchs04} for an example.
 \begin{figure}
 \centering	
 \begin{tikzpicture}
\shadedraw[left color=white!35!black, right color=white!90!black,thick]
(-4,0) to[out=0,in=180] (4,0)
to[out=100,in=260] (4,3)
to[out=180,in=0] (-4,3)
to[out=260,in=100] (-4,0);
\draw[dashed] (-4,3) to[out=280,in=80] (-4,0);
\shadedraw[left color=white!35!black, right color=white!90!black,thick]
 (4,0)to[out=100,in=260] (4,3)
to[out=280,in=80] (4,0);
\filldraw[fill=white, very thick] (-2,1.5) ellipse (1 and .5);
\filldraw[fill=white, very thick] (2,1.5) ellipse (1 and .5);

\draw (-2,1) to[out=260,in=100] (-2,0);
\draw (2,1) to[out=260,in=100] (2,0);
\draw (-2,3) to[out=260,in=100] (-2,2);
\draw (2,3) to[out=260,in=100] (2,2);
\draw (0,3) to[out=260,in=100] (0,0);
\draw (-2,1.5) ellipse (1.5 and 1);
\draw (2,1.5) ellipse (1.5 and 1);
\draw (-2,-.3) node{$\gamma_{1}$};
\draw (2,-.3) node{$\gamma_{2}$};
\draw (-2,3.3) node{$\beta_{1}$};
\draw (2,3.3) node{$\beta_{2}$};
\draw (-.8,.5) node{$\delta_{1}$};
\draw (.8,.5) node{$\delta_{2}$};
\draw (0.1,1.5) node{$\alpha$};

 \end{tikzpicture}
 \caption{Catching $\alpha$ in a $S_{0,4}$.}
 \label{Catchs04}
 \end{figure}
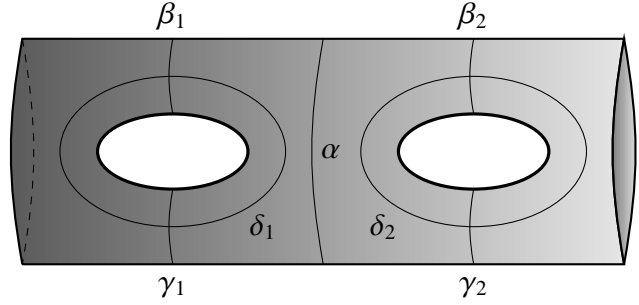

  By construction and lemma \ref{pairofpants}, $\phi(\alpha)$ and $h(\alpha)$ are contained in the $S_{0,4}$ subsurface bounded by $\phi(\beta_{1})$, $\phi(\gamma_{1})$, $\phi(\beta_{2})$ and $\phi(\gamma_{2})$ (recall that $\phi(\beta_{i}) = h(\beta_{i})$ and $\phi(\gamma_{i}) = h(\gamma_{i})$ for $i = 1,2$ since they are nonseparating curves). Even more, since $i(\alpha,\delta_{i}) = 0$ for $i=1,2$ then $\phi(\alpha)$ and $h(\alpha)$ must be contained in the annulus formed by cutting the aforementioned $S_{0,4}$ subsurface along the arcs of $\phi(\delta_{i}) = h(\delta_{i})$ for $i=1,2$; therefore $\phi(\alpha) = h(\alpha)$.
 \item If $\alpha$ is an outer curve, then let $P$ be a pants decomposition such that the peripheral pairs of $P$ bounding the same boundary components as $\alpha$, are consecutive to one another (similar to the proof of lemma \ref{peripheralpairs}), and $\alpha$ intersects only one curve in $P$ (namely $\beta$); let also $\gamma$ be a nonseparating curve  that intersects each curve in the peripheral pairs bounding the same boundary component as $\alpha$ only once while being disjoint from $\alpha$. Figure \ref{s04bis} illustrates this situation.
 \begin{figure}
 \centering
 \begin{tikzpicture}
\shadedraw[left color=white!35!black, right color=white!90!black,thick]
(0,0) to[out=0,in=180] (4,0)
to[out=100,in=260] (4,3)
to[out=180,in=0] (0,3)
to[out=260,in=100] (0,2)
to[out=0,in=90] (.5,1.5)
to[out=270,in=0] (0,1)
to[out=260,in=100] (0,0);
\draw[dashed] (0,1) to[out=280,in=80] (0,0); 
\draw[dashed] (0,3) to[out=280,in=80] (0,2); 
\shadedraw[left color=white!35!black, right color=white!90!black,thick]
 (4,0)
to[out=100,in=260] (4,3)
to[out=280,in=80] (4,0);
\filldraw[fill=white] (2.5,1.5) circle (.5);

\draw (1.5,3) to[out=260,in=100] (1.5,0);
\draw (.5,1.5) to[out=10,in=170] (2,1.5);
\draw (2.5,1) to[out=260,in=100] (2.5,0);
\draw (2.5,3) to[out=260,in=100] (2.5,2);
\draw (2.5,1.5) ellipse (1 and 1);
\draw (.8,1.2) node{$\beta$};
\draw (1.2,2.5) node{$\alpha$};
\draw (3.2,2.5) node{$\gamma$};
 \end{tikzpicture}
\caption{Catching $\alpha$ again in a $S_{0,4}$.}
\label{s04bis}
 \end{figure}
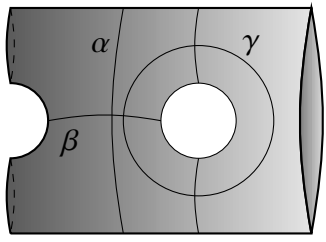
 Due to $\phi$ being an isomorphism, $\phi(\alpha)$ will intersect $\phi(\beta)$ and be disjoint of every other curve in $P$. Using that and lemma \ref{peripheralpairs}, we know that $\phi(\alpha)$ and $h(\alpha)$ are contained in the $S_{0,4}$ subsurface bounded by two boundary components of $S$ and the images of the adjacent curves in $\adj{P}$ of $\beta$; even more, $\phi(\alpha)$ and $h(\alpha)$ must be contained in the pair of pants resulting from cutting the aforementioned $S_{0,4}$ subsurface that contains them along the arc of $\phi(\gamma) = h(\gamma)$. Since there is only one curve in this pair of pants which is an essential curve of $S$, then $\phi(\alpha) = h(\alpha)$.
 \item Let $S_{1}$ and $S_{2}$ be the two connected components of $S_{\alpha}$ and suppose that $S_{1}$ has genus zero  and $n^{\prime} \geq 3$ boundary components. We can find the following: a finite sequence $\{\beta_{i}\}_{i=1}^{n^{\prime}-1}$ composed of outer curves, such that $i(\beta_{i}, \alpha) = 0$ for $i = 1,\ldots,n^{\prime}-1$, $i(\beta_{i},\beta_{i+1}) = 2$ for $i = 1, \ldots, n^{\prime} - 2$ and $i(\beta_{i},\beta_{j}) = 0$ for $j \notin \{i-1,i+1\}$; a pants decomposition $P$ (composed solely of nonseparating curves) of the infinite genus connected component of $S\backslash \{\alpha\}$; and finally, a curve $\gamma$ which intersects once the curves $\delta_{1}$ and $\delta_{2}$ forming the peripheral pair that bounds the boundary of $S_{2}$ induced by $\alpha$. Figure \ref{annulus} illustrates this situation.
\begin{figure}
\centering
\begin{tikzpicture}
\shadedraw[left color=white!35!black, right color=white!90!black, very thick]
(-2*.5,-4*.5) to[out=0,in=180] (6*.5,-4*.5)
to[out=100,in=260] (6*.5,4*.5)
to[out=180,in=0] (-2*.5,4*.5)
to[out=270,in=0] (-3*.5,3*.5)
to[out=0,in=90] (-2*.5,2.5*.5)
to[out=270,in=0] (-3*.5,2*.5)
to[out=280,in=80] (-3*.5,1*.5)
to[out=0,in=90] (-2*.5,0)
to[out=270,in=0] (-3*.5,-1*.5)
to[out=280,in=80] (-3*.5,-2*.5)
to[out=0,in=90] (-2*.5,-2.5*.5)
to[out=270,in=0] (-3*.5,-3*.5)
to[out=0,in=90] (-2*.5,-4*.5);

\shadedraw[left color=white!35!black, right color=white!90!black]
(6*.5,-4*.5)
to[out=100,in=260] (6*.5,4*.5) 
to[out=280,in=80] (6*.5,-4*.5); 

\shadedraw[right color=white!35!black, left color=white!90!black] 
(-2*.5,4*.5)
to[out=270,in=0] (-3*.5,3*.5)
to[out=90,in=180] (-2*.5,4*.5);

\shadedraw[right color=white!35!black, left color=white!90!black] 
(-3*.5,2*.5)
to[out=280,in=80] (-3*.5,1*.5)
to[out=120,in=240] (-3*.5,2*.5);

\shadedraw[right color=white!35!black, left color=white!90!black] 
 (-3*.5,-1*.5)
to[out=280,in=80] (-3*.5,-2*.5)
to[out=120,in=240] (-3*.5,-1*.5);

\shadedraw[right color=white!35!black, left color=white!90!black] 
 (-3*.5,-3*.5)
to[out=0,in=90] (-2*.5,-4*.5)
to[out=180,in=270] (-3*.5,-3*.5);

\filldraw[fill=white,very thick] (1.5,0) circle (.25);


\draw (-.5,2) to[out=270,in=45] (-1.03,.1);
\draw (-.5,-2) to[out=90,in=315] (-1.03,-.1);
\draw (-1,1.25) to[out=315,in=45] (-1,-1.25);
\draw (.5,2) to[out=260,in=100] (.5,-2);
\draw (1.5,2) to[out=260,in=100] (1.5,.25);
\draw (1.5,-.25) to[out=260,in=100] (1.5,-2);
\draw (1.5,0) ellipse (.5 and 1);
\draw(1*.5,0) node{$\alpha$};
\draw(2.2,0) node{$\gamma$};
\draw(1.8,1.5) node{$\delta_{1}$};
\draw(1.8,-1.5) node{$\delta_{2}$};
\draw(-.2,1.5) node{$\beta_{1}$};
\draw(-.25,0) node{$\beta_{2}$};
\draw(-.2,-1.5) node{$\beta_{3}$};
\end{tikzpicture}
\caption{Catching $\alpha$ in an annulus.}
\label{annulus}
\end{figure}
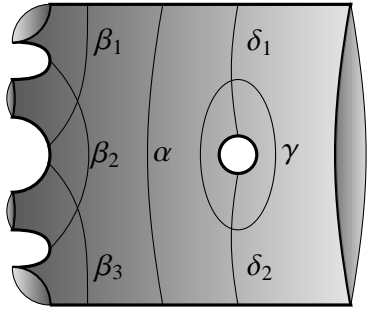

 Given that isomorphism of $\comp{S}$ send outer curves to outer curves, part $(2)$ of this proof, the fact that $\phi(\alpha)$ and $h(\alpha)$ must both be essential curves and they must be different from every element of $\phi(\{\beta_{i}\}_{i=1}^{n^{\prime}}) \cup \phi(P) \cup \{\phi(\gamma)\}$; we can conclude that $\phi(\alpha)$ and $h(\alpha)$ must be contained in the annulus obtained by cutting $S$ along $\phi( \{\beta_{i} \}_{i=1}^{n^{\prime}}) \cup \phi(P) \cup \{\phi(\gamma)\}$. The boundary components of this annulus are formed by arcs of $\phi(\beta_{i})$ for $i=1,\ldots,n^{\prime}-1$, $\phi(\gamma)$, $\phi(\delta_{1})$ and $\phi(\delta_{2})$. Therefore $\phi(\alpha)=h(\alpha)$.
\end{enumerate}
\subsubsection{Proof of theorem \ref{TeoISO}.}
 Theorems \ref{injCS} and \ref{surCS} imply that $\Psi_{\comp{S}}$ is an isomorphism. From theorem \ref{surGS} we know that the natural  map:
\begin{equation}
\Psi_{\nons{S}}:{\rm MCG}^{*}(S)\to {\rm Aut}(\mathcal{N}(S))
\end{equation}
is surjective. Let us suppose $h_{1}, h_{2} \in \MCG{S}$ are such that $h_{1} \neq h_{2}$ and $\Psi_{\nons{S}}(h_{1}) = \Psi_{\nons{S}}(h_{2})$. Then since $\Psi_{\comp{S}}$ is injective we have that $\Psi_{\comp{S}}(h_{1}) \neq \Psi_{\comp{S}}(h_{2})$ even though their restrictions to $\nons{S}$ are the same. This implies that $\Psi_{\comp{S}}(h_{1})$ and $\Psi_{\comp{S}}(h_{2})$ differ in some separating curves. But given that the restrictions of $\Psi_{\comp{S}}(h_{1})$ and $\Psi_{\comp{S}}(h_{2})$ to $\nons{S}$ are the same, we can use the same technique as in the proof of theorem \ref{surCS}, for catching the separating curves in an annulus (or a pair of pants), which means $\Psi_{\comp{S}}(h_{1})(\alpha) = \Psi_{\comp{S}}(h_{2})(\alpha)$ for every separating curve $\alpha$. Thus we have reached a contradiction and therefore $\Psi_{\nons{S}}$ is injective, hence it is an isomorphism. We finish the proof by remarking that $\Psi_{\gs{S}}=\Phi\circ\Psi_{\nons{S}}$, where $\Phi$ is the isomorphism between $\Aut{\nons{S}}$ and $\Aut{\gs{S}}$  defined in (\ref{E:fiso}). \qed


\begin{remark}
Using theorem \ref{TeoISO} we can deduce that, for an infinite genus surface $S$ such that $\Ends{S}=\Endsg{S}$, every automorphism $\varphi$ of ${\rm MCG}^{*}(S)$ sending Dehn twists to Dehn twist must be an inner automorphism. The proof of this fact is taken verbatim from the proof of theorem 2, in \cite{Ivanov}. However, it is still unknown if, as in the compact case, every automorphism of ${\rm MCG}^{*}(S)$ sends Dehn twists to Dehn twists. 
\end{remark}

\section{Counterexamples}
 \label{S:CFE}
 In this section we show that theorem \ref{S1homeoS2} is not valid if the morphism between curve complexes is not an isomorphism. For that, let us first recall the notion of \emph{superinjective map}.
\begin{Def}[\sc Superinjectivity]
	\label{D:super}
 A simplicial map $\fun{f}{\comp{S_{1}}}{\comp{S_{2}}}$ is called superinjective if for any two vertices $\alpha$ and $\beta$  in $\comp{S_{1}}$ such that $i(\alpha,\beta)\neq 0$ we have that $i(f(\alpha),f(\beta))\neq 0$. 
\end{Def}

Every superinjective map is injective. For compact surfaces, we have the following theorem concerning superinjective maps.
\begin{Teo}\cite{Irmak}
Let $S$ be a closed, connected, orientable surface of genus at least 3. A simplicial map, $f:\comp{S}\to\comp{S}$, is superinjective if and only if $f$ is induced by an homeomorphism of $S$.
\end{Teo}
The following lemma shows that this result is not true for a large class of surfaces of infinite genus and, in this sense, theorem \ref{S1homeoS2} is optimal. 
\begin{Lema}
	\label{L:CE1}
Let $S$ be a surface such that $\Endsg{S}\neq\emptyset$. Then there exist a simplicial superinjective map $f:\comp{S}\to\comp{S}$ which is not surjective. 
\end{Lema}
\begin{proof}
This proof makes reference to figure \ref{F:elem2}.
Let $\alpha \in \mathcal{V}(\comp{S})$ be a separating curve. Without loss of generality we can think that $\alpha$ is contained in a subsurface $S_{i}$ in $[S_{1}\supseteq S_{2}\supseteq\ldots]\in\Endsg{S}$ where $i$ is large enough.
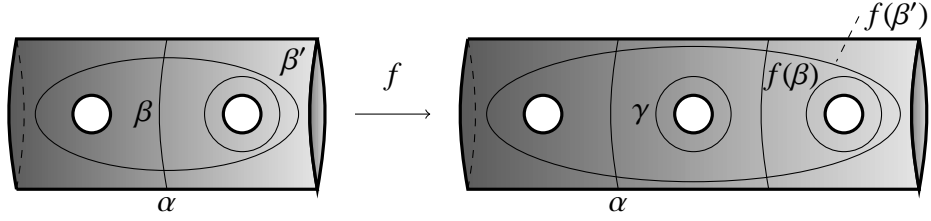
\begin{figure}
\centering
\begin{tikzpicture}

\shadedraw[left color=white!35!black, right color=white!90!black, very thick]
(-5,-1) to[out=0,in=180] (-1,-1)
to[out=100,in=260] (-1,1)
to[out=180,in=0] (-5,1)
to[out=260,in=100] (-5,-1);
\draw[dashed] (-5,-1) to[out=80,in=280] (-5,1);

\shadedraw[left color=white!35!black, right color=white!90!black, very thick]
(-1,1) to[out=260,in=100] (-1,-1)
to[out=80,in=280] (-1,1);

\shadedraw[left color=white!35!black, right color=white!90!black, very thick]
(1,1) to[out=0,in=180] (7,1)
to[out=260,in=100] (7,-1)
to[out=180,in=0] (1,-1)
to[out=100,in=260] (1,1);
\draw[dashed] (1,-1) to[out=80,in=280] (1,1);
\shadedraw[left color=white!35!black, right color=white!90!black, very thick]
(7,1) to[out=260,in=100] (7,-1)
to[out=80,in=280] (7,1);


\filldraw[fill=white, very thick] (-4,0) circle (.25);
\filldraw[fill=white, very thick] (-2,0) circle (.25);

\filldraw[fill=white, very thick] (2,0) circle (.25);
\filldraw[fill=white, very thick] (4,0) circle (.25);
\filldraw[fill=white, very thick] (6,0) circle (.25);



\draw (-3,1) to[out=260,in=100] (-3,-1);
\draw (-2,0) ellipse (.5 and .5);
\draw (-3,0) ellipse (1.75 and .75);

\draw (3,1) to[out=260,in=100] (3,-1);
\draw (5,1) to[out=260,in=100] (5,-1);
\draw (4,0) ellipse (.5 and .5);
\draw (6,0) ellipse (.5 and .5);
\draw (4,0) ellipse (2.75 and .9);

\draw[->] (-.5,0)--(.5,0);



\draw(-3,-1.2) node{$\alpha$};
\draw(-3.3,0) node{$\beta$};
\draw(-1.3,0.7) node{$\beta'$};


\draw(3,-1.2) node{$\alpha$};
\draw(3.3,0) node{$\gamma$};
\draw(5.3,0.5) node{$f(\beta)$};
\draw(6.7,1.3) node{$f(\beta')$};
\draw[dashed] (6.2,1.3) -- (5.9,0.7);
\draw(0,0.5) node{$f$};
\end{tikzpicture}
\caption{A superinjective but not surjective simplicial map.}
\label{F:elem2}
\end{figure}
	We describe $f$ topologically. Let $S_{1}$ and $S_{2}$ be the two connected components of $S_{\alpha}$. Cut  $S$ along $\alpha$ and then glue in a copy of $S_{1,2}$. This operation produces a new surface $S'=S_{1}\cup S_{2}\cup S_{1,2}$. Remark that $S$ is homeomorphic to $S'$ and that there is a natural inclusion map $f_{i}:S_{i}\hookrightarrow S'$, for $i=1,2$. If $\beta \in \mathcal{V}(\comp{S_{i}})$, then we define $f(\beta)=f_{i}(\beta)$ for $i=1,2$.  On the other hand, if $\beta'$ intersects the curve $\alpha$ we define $f(\beta')$ as depicted in figure  \ref{F:elem2}. Clearly, $f$ is superinjecive but no essential curve properly contained in the copy of $S_{1,2}$ that we introduced is  in the image of $f$. Hence $f$ is not surjective and, in particular, $f$ cannot be induced by a class in ${\rm MCG}^{*}(S)$.
\end{proof}
We think that this result can be optimized in the following way.
\begin{Conj}
Let S be a surface such that $ \Endsg{S} \neq \emptyset$ and $\{\alpha_{1}, \ldots,\alpha_{n}\} \subset \comp{S}$ be simplex. Then there exists a  simplicial superinjective map $f:\comp{S}\to\comp{S}$ whose image does not intersect $\{\alpha_{1},\ldots,\alpha_{n}\}$.
\end{Conj}
The following result shows that the statement of theorem \ref{S1homeoS2} is not valid for superinjective maps.


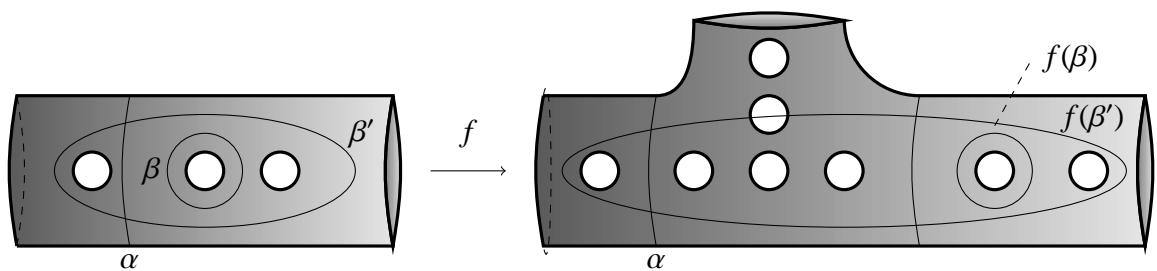
\begin{figure}
\centering
\begin{tikzpicture}
\shadedraw[left color=white!35!black, right color=white!90!black, very thick]
(-6,-1) to[out=0,in=180] (-1,-1)
to[out=100,in=260] (-1,1)
to[out=180,in=0] (-6,1)
to[out=260,in=100] (-6,-1);
\draw[dashed] (-6,-1) to[out=80,in=280] (-6,1);

\shadedraw[left color=white!35!black, right color=white!90!black, very thick]
(-1,1) to[out=260,in=100] (-1,-1)
to[out=80,in=280] (-1,1);


\shadedraw[left color=white!35!black, right color=white!90!black, very thick]
(1,1) to[out=0,in=180] (2.5,1)
to[out=0,in=270] (3,2)
to[out=350,in=190] (5,2)
to[out=270,in=180] (6,1)
to[out=0,in=180] (9,1)
to[out=260,in=100] (9,-1)
to[out=180,in=0] (1,-1)
to[out=100,in=260] (1,1);
\draw[dashed] (1,1) to[out=80,in=280] (1,-1);
\shadedraw[left color=white!35!black, right color=white!90!black, very thick]
(9,1) to[out=260,in=100] (9,-1)
to[out=80,in=280] (9,1);
\shadedraw[bottom color=white!35!black, top color=white!90!black, very thick]
(3,2) to[out=10,in=170] (5,2)
to[out=190,in=350] (3,2);



\filldraw[fill=white, very thick] (-5,0) circle (.25);
\filldraw[fill=white, very thick] (-3.5,0) circle (.25);
\filldraw[fill=white, very thick] (-2.5,0) circle (.25);


\filldraw[fill=white, very thick] (1.75,0) circle (.25);
\filldraw[fill=white, very thick] (3,0) circle (.25);
\filldraw[fill=white, very thick] (4,0) circle (.25);
\filldraw[fill=white, very thick] (5,0) circle (.25);

\filldraw[fill=white, very thick] (4,.75) circle (.25);
\filldraw[fill=white, very thick] (4,1.5) circle (.25);

\filldraw[fill=white, very thick] (7,0) circle (.25);
\filldraw[fill=white, very thick] (8.25,0) circle (.25);



\draw (-4.5,1) to[out=260,in=100] (-4.5,-1);

\draw (-3.5,0) ellipse (.5 and .5);

\draw (-3.5,0) ellipse (2 and .75);

\draw (2.5,1) to[out=260,in=100] (2.5,-1);
\draw (6,1) to[out=260,in=100] (6,-1);

\draw (7,0) ellipse (.5 and .5);
\draw (5,0) ellipse (3.75 and .75);



\draw (-4.5, -1.2) node{$\alpha$};
\draw (-4.2, 0) node{$\beta$};
\draw (-1.4, .5) node{$\beta'$};


\draw (2.5, -1.2) node{$\alpha$};
\draw (8.3, .7) node{$f(\beta')$};
\draw[dashed] (7,.6)--(7.5,1.5);

\draw (8, 1.5) node{$f(\beta)$};
\draw(0,0.5) node{$f$};
\draw[->] (-.5,0)--(.5,0);

\end{tikzpicture}
\caption{A superinjecive map between two nonhomeomorphic surfaces.}
\label{s04}
\end{figure}

\begin{Lema}
There exist uncountably many examples of pairs of nonhomeomorphic infinite genus surfaces $S_{1}$ and $S_{2}$ for which there exists a superinjective map $f:\comp{S_{1}}\to\comp{S_{2}}$.
\end{Lema}
\begin{proof}
The arguments are similar to those of the proof of lema \ref{L:CE1}. Let $S_{1}$ be the Loch Ness monster and $\alpha\in\comp{S_{1}}$ be a separating curve. Let $S$ be your favorite infinite genus surface and suppose that $S$ has at least two boundary components. We describe $f$ topologically. Cut  $S_{1}$ along $\alpha$ and then glue in a copy of $S$ as indicated in figure \ref{s04}. This produces $S_{2}$. The rest of the proof is analogous to the proof of lemma \ref{L:CE1}.\\
\end{proof}

\end{document}